%% file: main.tex
\documentclass{amsart}
\input{structural/Preamble}
\begin{document}

\input{structural/Title}
\input{structural/TOC}

\newcommand{\inv}{\mathrm{inv}}
\newcommand{\Lat}{\mathrm{Lat}}
\newcommand{\Typ}{\mathrm{Typ}}
\newcommand{\Rel}{\mathrm{Rel}}
\newcommand{\pair}[2]{\left\{#1,#2\right\}}
\newcommand{\cfactor}{\gamma}

\newenvironment{case}[1]{\begin{description}[leftmargin=0pt]\item[Case \boldmath\ensuremath{#1}]}{\end{description}}
\newcommand{\Ugen}{\ensuremath{\mathrm{(uH)}}\xspace}
\newcommand{\Ogen}{\ensuremath{\mathrm{(S)}}\xspace}
\newcommand{\SPgen}{\ensuremath{\mathrm{(A)}}\xspace}

\newcommand{\smallarray}[2]{\ifx\dummy#1\dummy #2\else\ifx\dummy#2\dummy #1\else\begin{array}{@{\,}c@{\,}}#1\\[-5pt]#2\end{array}\fi\fi}
\newcommand{\smallarrayl}[2]{\ifx\dummy#1\dummy #2\,\else\ifx\dummy#2\dummy #1\,\else\begin{array}{@{}c@{\,}}#1\\[-5pt]#2\end{array}\fi\fi}
\newcommand{\smallarrayr}[2]{\ifx\dummy#1\dummy\,#2\else\ifx\dummy#2\dummy\,#1\else\begin{array}{@{\,}c@{}}#1\\[-5pt]#2\end{array}\fi\fi}
\newcommand{\smallarraylr}[2]{\ifx\dummy#1\dummy\,#2\,\else\ifx\dummy#2\dummy\,#1\,\else\begin{array}{@{}c@{}}#1\\[-5pt]#2\end{array}\fi\fi}

\newcommand{\TypeCC}[2]{\left(\smallarraylr{#1}{#2}\right)}
\newcommand{\TypeC}[4]{\left(\smallarrayl{#1}{#2},\smallarrayr{#3}{#4}\right)}
\newcommand{\Typec}[6]{\left(\smallarrayl{#1}{#2},\smallarray{#3}{#4},\smallarrayr{#5}{#6}\right)}
\newcommand{\Typecc}[8]{\left(\smallarrayl{#1}{#2},\smallarray{#3}{#4},\smallarray{#5}{#6},\smallarrayr{#7}{#8}\right)}
\newcommand{\Type}[2]{\left(\cdots,\smallarray{#1}{#2},\cdots\right)}
\newcommand{\Typee}[4]{\left(\cdots,\smallarray{#1}{#2},\smallarray{#3}{#4},\cdots\right)}
\newcommand{\Typeee}[6]{\left(\cdots,\smallarray{#1}{#2},\smallarray{#3}{#4},\smallarray{#5}{#6},\cdots\right)}
\newcommand{\typeC}[2]{\left(#1,#2\right)}
\newcommand{\typec}[3]{\left(#1,#2,#3\right)}
\newcommand{\typecc}[4]{\left(#1,#2,#3,#4\right)}
\newcommand{\type}[1]{\left(\cdots,#1,\cdots\right)}
\newcommand{\typee}[2]{\left(\cdots,#1,#2,\cdots\right)}

\input{1-Introduction}
\subsection*{Acknowledgments}
The author would like to thank Andrew Graham, Siddharth Mahendraker and Zhiyu Zhang for their interest and helpful discussions. The author is grateful to Wei Zhang for his guidance and suggestions. This work was supported in part by the NSF grant DMS--1901642, and in part by the NSF grant DMS--1440140, while the author was in residence at the Simons Laufer Mathematical Sciences Institute in Berkeley, California, during the Spring Semester of 2023.

\input{2-FiniteFields}

\input{3-Lattices}
\input{4-Straightening}

\input{5-HeckeStructure}

\begin{appendix}
\input{6-Appendix}

\end{appendix}

\newpage
\input{structural/Bibliography}
\input{output.ind}
\end{document}

%% file: structural/Preamble.tex
\usepackage{mathtools,amsthm}
\usepackage[letterpaper, margin=1.3in]{geometry}
\usepackage[utf8]{inputenc}
\usepackage[T2A,T1]{fontenc}
\usepackage{imakeidx}
\usepackage{enumitem}
\usepackage{mathrsfs}
\usepackage{tikz-cd}
\usepackage{amssymb}
\usepackage{xspace}
\usepackage{etoolbox}
\usepackage{stmaryrd}
\usepackage{todonotes}
\usepackage[pagebackref]{hyperref}
\usepackage[nameinlink]{cleveref}
\makeindex[columns=2,title={Index of Notation},options={-s structural/index_style.ist}]

\hypersetup{colorlinks={true},linkcolor={blue},urlcolor=blue}

\newtheorem{theorem}{Theorem}[section]
\newtheorem{lemma}[theorem]{Lemma}
\newtheorem{corollary}[theorem]{Corollary}
\newtheorem{proposition}[theorem]{Proposition}

\newtheorem{innercustomgeneric}{\customgenericname}
\providecommand{\customgenericname}{}
\newcommand{\newcustomtheorem}[2]{%
  \newenvironment{#1}[1]
  {%
   \renewcommand\customgenericname{#2}%
   \renewcommand\theinnercustomgeneric{\kern-0.3em ##1}%
   \innercustomgeneric
  }
  {\endinnercustomgeneric}
}

\newcustomtheorem{specialtheorem}{}
\crefname{specialtheorem}{}{}

\theoremstyle{definition}
\newtheorem{definition}[theorem]{Definition}

\theoremstyle{remark}
\newtheorem{remark}[theorem]{Remark}
\newtheorem*{remark*}{Remark}

\newtheorem{example}[theorem]{Example}
\newtheorem*{example*}{example}

\newcommand{\Z}{\mathbb{Z}}
\newcommand{\Q}{\mathbb{Q}}

\renewcommand{\C}{\mathbb{C}}
\newcommand{\F}{\mathbb{F}}
\newcommand{\m}{\mathfrak{m}}
\renewcommand{\O}{\mathcal{O}}
\newcommand{\rightiso}{\xrightarrow{\sim}}
\newcommand{\iso}{\simeq}
\newcommand{\xrightiso}[1]{\xrightarrow[\raisebox{1pt}{$\sim$}]{#1}}

\newcommand{\defeq}{\vcentcolon=}
\newcommand{\eqdef}{=\vcentcolon}
\newcommand{\isequal}{\stackrel{?}{=}}
\newcommand{\qbinom}[3]{\genfrac{[}{]}{0pt}{}{#1}{#2}_{#3}}

%% file: structural/Title.tex
\title{Spherical functions of symmetric forms and a conjecture of Hironaka}
\author{Murilo Corato-Zanarella}

\begin{abstract}
For all $r\ge1,$ we verify the following conjecture of Hironaka \cite{Hironaka2}: for a $p$-adic field $F$ with $p$ odd, the space of spherical functions of $\mathrm{Sym}_{r\times r}(F)\cap\mathrm{GL}_r(F)$ is free of rank $4^r$ over the Hecke algebra.
\end{abstract}
\maketitle

%% file: structural/TOC.tex
\setcounter{tocdepth}{2}
\let\oldtocsection=\tocsection
\let\oldtocsubsection=\tocsubsection
\let\oldtocsubsubsection=\tocsubsubsection
\renewcommand{\tocsection}[2]{\hspace{0em}\oldtocsection{#1}{#2}}
\renewcommand{\tocsubsection}[2]{\hspace{1em}\oldtocsubsection{#1}{#2}}
\renewcommand{\tocsubsubsection}[2]{\hspace{2em}\oldtocsubsubsection{#1}{#2}}
\tableofcontents

%% file: 1-Introduction.tex
\section{Introduction}
\subsection{Main result}
Let $F$ be a finite extension of $\Q_p$ for some prime $p>2.$ Let $r\ge1$ be an integer. We consider the space
\begin{equation*}
    X\defeq\mathrm{Sym}_{r\times r}(F)\cap\mathrm{GL}_r(F)
\end{equation*}
of symmetric nondegenerate $r\times r$ matrices with entries in $F.$ This is equipped with a right action of $K\defeq\mathrm{GL}_r(\O_F)$ via $x\cdot k= k^\intercal x k.$ For a given coefficient ring $R,$ we consider the space of $R$-valued \emph{spherical functions} $\mathcal{S}(X/K,R),$ that is, $R$-valued compactly supported and locally constant $K$-invariant functions $X\to R.$ This is a module for the ($R$-valued) Hecke algebra $\mathcal{H}(G,K,R)$ of $G\defeq\mathrm{GL}_r(F)$ with respect to $K.$

Hironaka studied the spherical functions $\mathcal{S}(X/K,\C)$ in the series of papers \cite{Hironaka1,Hironaka2,Hironaka3}. She conjectured that $\mathcal{S}(X/K,\C)$ is free of rank $4^r$ over $\mathcal{H}(G,K,\C),$ and proved such conjecture for $r\le 2.$ In this paper, we prove this conjecture in general.
\begin{specialtheorem}{Theorem A}\label{ThmA}
Let $R$ be a ring where $2$ is invertible. Then $\mathcal{S}(X/K,R)$ is a free $\mathcal{H}(G,K,R)$-module of rank $4^r.$
\end{specialtheorem}

Furthermore, in \Cref{explicitBasis} we provide an explicit generating set. We also treat at the same time the case of nondegenerate Hermitian matrices over unramified quadratic extensions, as well as the case of nondegenerate alternating matrices. These cases have long been understood by Hironaka \cite{HironakaHerm} and Hironaka--Sato \cite{HironakaAlt}, but our method yields a different proof. Moreover, our method allow us to study the spaces $\mathcal{S}(X/K,R)$ with any coefficient ring, not just with $R=\C.$

The space $X$ is an example of a \emph{spherical variety}. More generally, consider $G$ a connected reductive group over a local field and $\check{G}$ its Langlands dual group. For a spherical $G$-variety $X,$ the relative Langlands program predicts (roughly) that there exists a certain \emph{dual object} $\check{M}$ equipped with an action of $\check{G}$ such that there is a matching of data from $X$ and data from $\check{M}$ reminiscent of mirror symmetry and other physical dualities, as explored by Ben-Szi--Sakellaridis--Venkatesh \cite{BZSV}. Under certain assumptions considered by Sakellaridis \cite{Sakellaridis}--which include $G$ being split and $X$ having no roots of type $N$--this dual object $\check{M}$ should be controlled by a certain subgroup $\check{G}_X\subseteq\check{G}$ (which is conjecturally isogenous to the subgroup constructed by Gaitsgory--Nadler \cite{Gaitsgory-Nadler}). In this case, \cite[Theorem 1.4.1]{Sakellaridis} proves that we indeed have an identification
\begin{equation}\label{SphericalTransform}
    \mathcal{S}(X/K)\iso R(\check{G}_X)
\end{equation}
of $\mathcal{S}(X/K)$ with the representation ring $R(\check{G}_X)$ of $\check{G}_X$ in such a way that the action of $\mathcal{H}(G,K,\C)$ is given by the composition of the Satake transform $\mathcal{H}(G,K,\C)\iso R(\check{G})$ with the map $R(\check{G})\to R(\check{G}_X)$ induced by $\check{G}_X\subseteq\check{G}.$

The case considered in our \Cref{ThmA} is not covered by the above results, as $X$ has roots of type $N.$ Since \Cref{ThmA} is a consequence of an explicit description of $\mathcal{S}(X/K,R)$ as a $\mathcal{H}(G,K,R)$-module (\Cref{inducedThm}), it should be possible to obtain an analogue of \eqref{SphericalTransform}. One expects, from a conjecture of Jacquet \cite{Jacquet}, that the dual object $\check{M}$ is controlled by the dual group of a double cover $\widetilde{\mathrm{GL}}_r(F)$ of $\mathrm{GL}_r(F).$ However, it is unclear to the author how to make a precise conjecture.

We note that our explicit description of the $\mathcal{H}(G,K,R)$-module $\mathcal{S}(X/K,R)$ is a consequence of the general work of Sakellaridis in the settings where it applies, as we spell out in \Cref{Section-Appendix}. As mentioned above, this requires $X$ to not have roots of type $N$ and for $G$ to be split. In a companion paper \cite{CZ}, we build upon the present results to give similar explicit descriptions of spherical functions in certain symmetric varieties of ``Friedberg--Jacquet type'', including more new examples which contain roots of type $N.$

\subsection{Strategy and structure of the paper}
For ease of exposition, we outline the strategy in the Hermitian case. We refer the reader to the body of the paper for the corresponding details in the symmetric and alternating cases.

Let $F/F_0$ be a quadratic unramified extension with uniformizer $\varpi\in \O_{F_0},$ and we denote $X=\mathrm{Herm}_{r\times r}(F)\cap\mathrm{GL}_r(F)$ to be the space of nondegenerate $F/F_0$-Hermitian matrices. If $V=F^r$ an $r$-dimensional vector space over $F,$ then we naturally have
\begin{equation*}
    X=U(V_+)\backslash G\ \sqcup \ U(V_-)\backslash G,\quad\text{and thus}\quad X/K=U(V_+)\backslash G/K\ \sqcup\  U(V_-)\backslash G/K,
\end{equation*}
where $V_+=(V,\langle\cdot,\cdot\rangle_+)$ and $V_-=(V,\langle\cdot,\cdot\rangle_-)$ are the two non-isomorphic nondegenerate Hermitian forms on $V.$ Each term $U(V_\pm)\backslash G/K$ can be interpreted in a combinatorial fashion: $G/K$ is naturally the set of lattices $\Lambda$ in $V,$ and its $U(V_\pm)$-orbits are encoded by the relative position of $\Lambda$ and $\Lambda^{\vee_\pm}\defeq\{x\in V\colon\langle x,\Lambda\rangle_{\pm}\subseteq\O_F\}.$ Putting both terms together, this gives a bijection
\begin{equation*}
    X/K\rightiso\Typ^0_r\defeq\{(e_1,\ldots,e_r)\in\Z^r\colon e_1\ge\cdots\ge e_r\}.
\end{equation*}
This is a relative Cartan decomposition of $X,$ and under this lattice-theoretic interpretation, the action of the Hecke algebra on $\mathcal{S}(X/K,R)$ gets translated to certain lattice-counting questions.

In general, these lattice-counting questions can be very complicated, but we can fully compute them for the \emph{minuscule} Hecke operators. The main observation that drives this paper forward is that these can be succinctly described in terms of certain \emph{straightening relations} which we outline below. Since the minuscule Hecke operators generate $\mathcal{H}(G,K,R),$ this gives an explicit way to succinctly describe the action of the \emph{entire} Hecke algebra.

For $0\le k\le r,$ let $\mu_k\defeq(\underbrace{1,\ldots,1}_k,0,\ldots,0)$ and denote $T_k=\mathrm{char}(K\varpi^{\mu_k} K)\in\mathcal{H}(G,K,R)$ to be the minuscule Hecke operators. We will prove the following description for its adjoint $T_k^*\colon R[\Typ^0]\to R[\Typ^0]$. i) If $e\in\Typ^0$ is such that $e_i\ge e_{i+1}+2$ for all $1\le i<r,$ then
\begin{equation*}
    T_k^*(e)=\sum_{\varepsilon\in\{0,1\}^r}q^{\inv(\varepsilon)}\cdot (e_1+2\varepsilon_1,\ldots,e_r+2\varepsilon_r),
\end{equation*}
ii) In general, denote $\Typ=\Z^r$ and define $\Delta_k\colon R[\Typ]\to R[\Typ]$ by the same formula above; then we construct a certain submodule $\Rel\subseteq\Z[\Typ]$ of \emph{straightening relations} satisfying the following properties: 1) $R[\Typ]/\Rel=R[\Typ^0],$ 2) $\Delta_k$ preserves $\Rel$ and 3) the endomorphism of $R[\Typ]/\Rel$ induced by $\Delta_k$ is precisely $T_k^*.$ In this Hermitian case, the straightening relations are generated by the following elements for all $a,b\in\Z$ with $a<b$\footnote{We are assuming $2\in R^\times$ for simplicity here.}
\begin{equation*}
\begin{split}
    &\typee{a}{b}-\typee{a+1}{b-1}\\
    &\quad\quad-(-\#\O_{F_0}/\varpi\O_{F_0})^{b-a-1}\left(\typee{b}{a}-\typee{b-1}{a+1}\right).
\end{split}
\end{equation*}

Such a description will allow us to study the $\mathcal{H}(G,K,R)$-module structure of $\mathcal{S}(X/K,R).$ Parallel to \cref{ThmA}, we can use this description to see that, in this Hermitian case, $\mathcal{S}(X/K,R)$ is a free $\mathcal{H}(G,K,R)$-module with basis
\begin{equation*}
    \left\{(e_1,\ldots,e_r)\in\Typ^0\colon e_1-e_2,\ldots,e_{r-1}-e_r,e_r\in\{0,1\}\right\}.
\end{equation*}

In \Cref{FiniteFieldSection}, we collect certain subspace-counts over finite fields, which are used in \Cref{MainLemmaSection} where we solve the general lattice-counting question related to minuscule Hecke operators. In \Cref{StraighteningSection}, we use the above results to give the explicit description of the Hecke action in terms of straightening relations. In \Cref{HeckeStructureSection}, we analyze the Hecke structure of the module of spherical functions. In addition to proving Hironaka's conjecture on the symmetric case, we also show how we can recover the spherical transforms of Hironaka resp.\ Hironaka--Sato in the Hermitian resp.\ alternating cases from our methods.

\subsection{Notations}
Throughout the whole paper, we will consider three cases: \Ugen, \Ogen and \SPgen, standing for \emph{unramified Hermitian}, \emph{symmetric} and \emph{alternating}. The three cases will be completely independent of each other.

\subsubsection*{$p$-adic fields}
We will denote $F_0$ to be a $p$-adic local field for some prime $p.$\index{p@$p$} We will consider\index{F@$F_0,\ F$}
\begin{case}{\Ugen}
$F/F_0$ an unramified quadratic extension,
\end{case}
\begin{case}{\Ogen}
$F=F_0,$ and we assume $p$ is odd,
\end{case}
\begin{case}{\SPgen}
$F=F_0.$
\end{case}
We denote by $\O_F$ and $\O_{F_0}$ the corresponding rings of integers, with maximal ideals $\m_F$ and $\m_{F_0}.$ We let $q\defeq\#\O_F/\m_F$ and $q_0\defeq\#\O_{F_0}/\m_{F_0}.$\index{q@$q_0,\ q$} Fix $\varpi\in\m_{F_0}$\index{1pi@$\varpi$} an uniformizer.

We denote $\mathrm{Sign}$\index{Sign@$\mathrm{Sign}$} to be the set
\begin{case}{\Ugen}
$\mathrm{Sign}=\{1\},$
\end{case}
\begin{case}{\Ogen}$\mathrm{Sign}=\O_F^\times/(\O_F^\times)^2=\F_q^\times/(\F_q^\times)^2=\{\pm1\},$
\end{case}
\begin{case}{\SPgen}
$\mathrm{Sign}=\{1\},$
\end{case}
and $\mathrm{sign}\colon\F_q^\times\to\mathrm{Sign}$ and $\mathrm{sign}\colon\O_F^\times\to\mathrm{Sign}$\index{sign@$\mathrm{sign}$} to be the natural maps. We also denote $\epsilon\defeq\mathrm{sign}(-1).$\index{1epsilon@$\epsilon$}

For $A\in\mathrm{Mat}_{a\times b}(F),$ we denote $A^*\in\mathrm{Mat}_{b\times a}(F)$ to be\index{0()*@$(\cdot)^*$}
\begin{case}{\Ugen}
$A^*=\overline{(A^\intercal)},$ where $\overline{(\cdot)}\colon F\to F$ is the nontrivial element of $\mathrm{Gal}(F/F_0),$
\end{case}
\begin{case}{\Ogen}
$A^*=A^\intercal,$
\end{case}
\begin{case}{\SPgen}
$A^*=-A^\intercal.$
\end{case}

We also denote
\index{1gamma@$\cfactor$}
\begin{equation*}
\cfactor\defeq \begin{cases}
    1 & \text{in the cases \Ugen, \Ogen,} \\
    2 & \text{in the case \SPgen.}
\end{cases}
\end{equation*}

\subsubsection*{Combinatorics}
For $n\in\Z_{\ge0},$ the $q$-Pochammer symbol and falling factorials are
\index{0(x;y)n@$(x;y)_n,\ (x)_n$}
\begin{equation*}
    (x;y)_n\defeq\prod_{i=1}^n(1-xy^{i-1}),\quad\text{and}\quad (x)_n\defeq(-1)^n(x;x)_n=\prod_{i=1}^n(x^i-1).
\end{equation*}
The $q$-analogues of binomial coefficients are
\index{0(nm)lambda@$\qbinom{n}{m}{\lambda}$}
\begin{equation*}
    \qbinom{n}{m}{\lambda}=\frac{(n)_\lambda}{(m)_\lambda\cdot(n-m)_\lambda}\quad\text{for integers }0\le m\le n.
\end{equation*}
For $n,m\in\Z,$ we will use the convention that $\qbinom{n}{m}{\lambda}=0$ if either $m<0,$ $m>n$ or $n<0.$ More generally, if $n,m_1,\ldots,m_r$ are integers with $n=m_1+\cdots+m_r,$ we also consider the $\lambda$-multinomials given by
\begin{equation*}
    \qbinom{n}{m_1,\ldots,m_r}{\lambda}=\frac{(n)_\lambda}{(m_1)_\lambda\cdots(m_r)_\lambda}\quad\text{if }m_1,\ldots,m_r\ge0,
\end{equation*}
and given by $0$ otherwise.

For $I$ a finite ordered set and $f\colon I\to\Z$ a function, we denote
\index{inv@$\inv,\ \widetilde{\inv}$}
\begin{equation*}
    \inv(f)\defeq\#\{(i,j)\in I^2\colon i<j\text{ and }f(i)>f(j)\}\quad\text{and}\quad \widetilde{\inv}(f)\defeq\sum_{\substack{(i,j)\in I^2\\ i<j}}\max(0,f(i)-f(j)).
\end{equation*}
For a sequence of integers $e\in\Z^n,$ and an integer $m\in\Z,$ we denote
\index{1lambdame@$\lambda_m(e)$}\index{1Sigmae@$\Sigma(e)$}
\begin{equation*}
    \lambda_m(e)\defeq\#\{1\le i\le n\colon e_i=m\},\quad\text{and}\quad\Sigma(e)\defeq\sum_{i=1}^ne_i.
\end{equation*}

%% file: 2-FiniteFields.tex
\section{Combinatorics over finite fields}\label{FiniteFieldSection}
Let $V$ be a finite dimensional $\F_q$-vector space. We consider nondegenerate pairings $\langle\cdot,\cdot\rangle\colon V\times V\to \F_q$ which are
\begin{case}{\Ugen}
$\F_q/\F_{q_0}$-Hermitian,
\end{case}
\begin{case}{\Ogen}
symmetric,
\end{case}
\begin{case}{\SPgen}
alternating.
\end{case}

The following is well-known.
\begin{proposition}\label{typdef}
    The set of isomorphism classes of nondegenerate $(V,\langle\cdot,\cdot\rangle)$ is in bijection with $\{(0,1)\}\sqcup(\Z_{>0}\times\mathrm{Sign})$ via $\mathrm{typ}(V,\langle\cdot,\cdot\rangle)\defeq(d,\chi)$ where $\chi=\mathrm{sign}(\det(V,\langle\cdot,\cdot\rangle))$\index{typ(V,<>)@$\mathrm{typ}(V,\langle\cdot,\cdot\rangle)$} and
    \begin{equation*}
        d\defeq\frac{1}{\cfactor}\dim_{\F_q}V=\begin{cases}
            \dim_{\F_q}V & \text{in the cases \Ugen, \Ogen,} \\
            \frac{1}{2}\dim_{\F_q}V & \text{in the case \SPgen.}
        \end{cases}
    \end{equation*}
\end{proposition}
\begin{remark}
Note that in case \Ogen, we are using the determinant rather than the usual discriminant $\mathrm{disc}(V)\defeq(-1)^{\binom{\dim V}{2}}\det(V).$ This will make the results on the following subsection less clean, but will make the later sections much cleaner due to the multiplicativity of the determinant.
\end{remark}

We will collect formulas for the following two quantities:
\begin{definition}\label{typFinite}
For $(V,\langle\cdot,\cdot\rangle)$ as above with $\mathrm{typ}(V)=(r,\chi),$ we consider:
\begin{enumerate}
    \item Given $0\le b\le r$ we denote
    \index{S(b,r,chi)@$S\TypeC{b}{}{r}{\chi}$}
    \begin{equation*}
        S\TypeC{b}{}{r}{\chi}\defeq\#\{W\subseteq V\colon \dim W=b,\ W\subseteq W^\perp\}
    \end{equation*}
    to be the number of $b$-dimensional isotropic subspaces of $V.$
    \item Given $0\le a\le r$ and $\eta\in\mathrm{Sign},$ we denote
    \index{R(a,eta,r,chi)@$R\TypeC{a}{\eta}{r}{\chi}$}
    \begin{equation*}
        R\TypeC{a}{\eta}{r}{\chi}\defeq\#\{W\subseteq V\colon W\text{ is nondegenerate, }\mathrm{typ}(W)=(a,\eta)\}.
    \end{equation*}
\end{enumerate}
\end{definition}

\begin{lemma}\label{Hprop}
Given $(V,\langle\cdot,\cdot\rangle)$ with $\mathrm{typ}(V)=(r,\chi),$ let $H(r,\chi)\subseteq\mathrm{GL}(V)$ be the subgroup of automorphisms which preserve $\langle\cdot,\cdot\rangle.$ Then we have
\begin{equation*}
    \#H(r,\chi)=\begin{cases}\displaystyle(-q_0)^{\binom{r}{2}}(r)_{-q_0}&\text{in the case \Ugen,}\\\displaystyle 2q^{\lfloor r/2\rfloor\cdot\lfloor(r-1)/2\rfloor}\frac{(\lfloor r/2\rfloor)_{q^2}}{e(r,\chi)}&\text{in the case \Ogen,}\\\displaystyle q^{r^2}(r)_{q^2}&\text{in the case \SPgen.}\end{cases}
\end{equation*}
Here we denote
\index{e(r,chi)@$e(r,\chi)$}
\begin{equation*}
    e(r,\chi)\defeq\begin{cases}q^{r/2}+\epsilon^{r/2}\chi&\text{if }r\text{ is even,}\\1&\text{if }r\text{ is odd.}\end{cases}
\end{equation*}
\end{lemma}
\begin{proof}
This is classical, see for example \cite{ClassicalGroups}.
\end{proof}

\begin{proposition}\label{Rprop}
We have
\begin{equation*}
    R\TypeC{a}{\eta}{r}{\chi}=\begin{cases}\displaystyle(-q_0)^{a(r-a)}\qbinom{r}{a}{-q_0}&\text{in the case \Ugen,}\\\displaystyle\frac{q^{\lfloor a(r-a)/2\rfloor}}{2}\frac{\lfloor r/2\rfloor_{q^2}}{\lfloor a/2\rfloor_{q^2}\lfloor(r-a)/2\rfloor_{q^2}}\frac{e(a,\eta)e(r-a,\chi\eta)}{e(r,\chi)}&\text{in the case \Ogen,}\\\displaystyle q^{2a(r-a)}\qbinom{r}{a}{q^2}&\text{in the case \SPgen.}\end{cases}
\end{equation*}
\end{proposition}
\begin{proof}
This follows at once from \Cref{Hprop} by noting that
\begin{equation*}
R\TypeC{a}{\eta}{r}{\chi}=\frac{\#H(r,\chi)}{\#H(a,\eta)\cdot\#H(r-a,\eta\chi)}.\qedhere
\end{equation*}
\end{proof}

\begin{proposition}\label{Sprop}
We have
\begin{equation*}
    S\TypeC{b}{}{r}{\chi}=\begin{cases}\displaystyle(-q_0;q_0^2)_b\qbinom{r}{2b}{-q_0}&\text{in the case \Ugen,}\\\displaystyle(-q;q)_b\qbinom{\lfloor r/2\rfloor}{b}{q^2}\frac{e(r-2b,\chi\epsilon^b)}{e(r,\chi)}&\text{in the case \Ogen,}\\\displaystyle(-q;q)_b\qbinom{r}{b}{q^2}&\text{in the case \SPgen.}\end{cases}
\end{equation*}
\end{proposition}
\begin{proof}
For the case \Ugen, this is \cite[Lemma 1.9.1]{Li-Zhang}. For the case \Ogen, this is \cite[Lemma 3.2.2]{Li-ZhangO}. The case \SPgen is \cite[Exercise 8.1(ii)]{ClassicalGroups} and we include a proof for completeness. We will follow the proof in \cite[Lemma 1.9.1]{Li-Zhang}.

The symplectic group $\mathrm{Sp}_{2r}(\F_q)$ acts transitively in the set of $b$-dimensional isotropic subspaces of $V.$ The stabilizer is a parabolic subgroup $P_b(\F_q).$ As an affine variety, we have
\begin{equation*}
    P_b\iso \mathrm{GL}_b\times\mathrm{Sp}_{2(r-b)}\times\mathbb{G}_a^{(b^2+b)/2}\times\mathbb{G}_a^{2b(r-b)},
\end{equation*}
and thus, using \Cref{Hprop},
\begin{equation*}
    S\typeC{b}{r}=\frac{q^{r^2}(r)_{q^2}}{q^{\binom{b}{2}}(b)_q q^{(r-b)^2}(r-b)_{q^2}q^{2b(r-b)}q^{(b^2+b)/2}}=(-q;q)_b\qbinom{r}{b}{q^2}.\qedhere
\end{equation*}
\end{proof}

\begin{theorem}\label{mainLemmaFinite}
Let $(V,\langle\cdot,\cdot\rangle)$ be a space with $\mathrm{typ}(V)=(r,\chi).$ Let $0\le m,n\le r$ and $\psi_1\in\mathrm{Sign},$ and denote
\begin{equation*}
    l\defeq r-n-\frac{2}{\cfactor}\cdot m,\quad\psi_2\defeq\chi\psi_1\epsilon^m.
\end{equation*}
Then the number of subspaces $N\subseteq V$ with $\dim(N\cap N^\perp)=m$ and $\mathrm{typ}(N/(N\cap N^\perp))=(n,\psi_1)$ is
\index{Q(n,psi1,m,l,psi2)@$Q\Typec{n}{\psi_1}{m}{}{l}{\psi_2}$}
\begin{equation*}
Q\Typec{n}{\psi_1}{m}{}{l}{\psi_2}\defeq\begin{cases}\displaystyle(-1)^m(-q_0)^{nl}\frac{(r)_{-q_0}}{(n)_{-q_0}(m)_{q_0^2}(l)_{-q_0}}&\text{in the case \Ugen,}\\
\displaystyle \frac{q^{\lfloor nl/2\rfloor}}{2}\frac{\lfloor r/2\rfloor_{q^2}}{\lfloor n/2\rfloor_{q^2}(m)_q\lfloor l/2\rfloor_{q^2}}\frac{e(n,\psi_1)e(l,\psi_2)}{e(r,\chi)}&\text{in the case \Ogen,}\\
\displaystyle q^{2nl}\frac{(r)_{q^2}}{(n)_{q^2}(m)_{q}(l)_{q^2}}&\text{in the case \SPgen.}\end{cases}
\end{equation*}
\end{theorem}
\begin{proof}
For such an $N,$ we denote $M\defeq N\cap N^\perp.$ Note that if $\dim(M)=m,$ then we have $\mathrm{sign}(\det(M^\perp/M))=\epsilon^m\chi.$ Counting $M$ first and then $N,$ we have that the number we want is
\begin{equation*}
    S\TypeC{m}{}{r}{\chi}\cdot R\TypeC{n}{\psi_1}{r-\frac{2}{\cfactor}\cdot m}{\epsilon^m\chi}.
\end{equation*}

In the case \Ugen, this is
\begin{equation*}
\begin{split}
    (-q_0;q_0^2)_m\qbinom{r}{2m}{-q_0}\cdot(-q_0)^{nl}\qbinom{r-2m}{n}{-q_0}=\ &(-q_0)^{nl}(-q_0;q_0^2)_m\qbinom{r}{n,2m,l}{-q_0}\\
    =\ &(-1)^m(-q_0)^{nl}\frac{(r)_{-q_0}}{(n)_{-q_0}(m)_{q^2}(l)_{-q_0}}.
\end{split}
\end{equation*}

In the case \Ogen, this is
\begin{equation*}
\begin{split}
    &(-q;q)_m\qbinom{\lfloor r/2\rfloor}{m}{q^2}\frac{e(r-2m,\chi\epsilon^m)}{e(r,\chi)}\cdot\frac{q^{\lfloor nl/2\rfloor}}{2}\frac{\lfloor r/2-m\rfloor_{q^2}}{\lfloor n/2\rfloor_{q^2}\lfloor l/2\rfloor_{q^2}}\frac{e(n,\psi_1)e(l,\chi\psi_1\epsilon^m)}{e(r-2m,\epsilon^m\chi)}\\
    &\quad\quad=\frac{q^{\lfloor nl/2\rfloor}}{2}\frac{\lfloor r/2\rfloor_{q^2}}{\lfloor n/2\rfloor_{q^2}(m)_q\lfloor l/2\rfloor_{q^2}}\frac{e(n,\psi_1)e(l,\psi_2)}{e(r,\chi)}.
\end{split}
\end{equation*}

In the case \SPgen, this is
\begin{equation*}
    (-q;q)_m\qbinom{r}{m}{q^2}\cdot q^{2nl}\qbinom{r-m}{n}{q^2}=q^{2nl}(-q;q)_m\qbinom{r}{n,m,l}{q^2}=q^{2nl}\frac{(r)_{q^2}}{(n)_{q^2}(m)_{q}(l)_{q^2}}.\qedhere
\end{equation*}
\end{proof}

%% file: 3-Lattices.tex
\section{Lattice counting}\label{MainLemmaSection}
We consider a finite dimensional $F$-vector space $V,$ equipped with a form $\langle\cdot,\cdot\rangle\colon V\times V\to F$ which is
\begin{case}{\Ugen}
Hermitian over $F/F_0,$
\end{case}
\begin{case}{\Ogen}
symmetric,
\end{case}
\begin{case}{\SPgen}
alternating.
\end{case}

Let $H\subseteq\mathrm{GL}(V,F)$ be the $p$-adic group of automorphisms that preserve $\langle\cdot,\cdot\rangle,$ and $\Lat(V)$\index{Lat(V)@$\Lat(V)$} be the set of $\O_F$-lattices of $V.$ Then $H$ acts on $\Lat(V)$ by $h\cdot\Lambda\defeq h(\Lambda).$ The orbits of this map are well understood, as we describe below.

\begin{definition}\label{Typzerodef}
We denote $\Typ^0\defeq\{(e^0,\chi^0)\}$\index{Typ0@$\Typ^0$} to be the set of pairs $(e^0,\chi^0)$ where $e^0\colon\Z\to\Z_{\ge0}$ and $\chi^0\colon\Z\to\mathrm{Sign}$ are such that
\begin{itemize}
    \item we have $e^0(i)=0$ for all but finitely many $i\in\Z,$
    \item if $i$ is such that $e^0(i)=0,$ then $\chi^0(i)=1.$
\end{itemize}
\end{definition}

For an $\O_F$-lattice $\Lambda,$ there always exist an orthogonal decomposition $V=\bigoplus_{i\in\Z}V^{(i)}_\Lambda$ such that $\Lambda^{(i)}\defeq\Lambda\cap V^{(i)}_\Lambda$ satisfies that $(\Lambda^{(i)})^\vee=\varpi^{-i}\Lambda^{(i)}.$ We equip $V^{(i)}_\Lambda$ with the form $\langle\cdot,\cdot\rangle^{(i)}\defeq\varpi^{-i}\langle\cdot,\cdot\rangle,$ for which $\Lambda^{(i)}$ becomes self-dual. Given such a decomposition, we consider the functions $e^0\colon\Z\to\Z_{\ge0}$ and $\chi^0\colon\Z\to\mathrm{Sign}$ given by $(e^0(i),\chi^0(i))\defeq\mathrm{typ}(\Lambda^{(i)}/\varpi\Lambda^{(i)},\langle\cdot,\cdot\rangle^{(i)})$ where $\mathrm{typ}$ was defined in \Cref{typdef}. That is,
\begin{equation*}
    e^0(i)\defeq\frac{1}{\cfactor}\dim_{F}V^{(i)}_\Lambda,\quad\chi^0(i)\defeq\mathrm{sign}(\det(\Lambda^{(i)},\langle\cdot,\cdot\rangle^{(i)})).
\end{equation*}
We will denote $\mathrm{typ}(\Lambda)\defeq(e^0,\chi^0).$\index{typ(Lambda)@$\mathrm{typ}(\Lambda)$}
\begin{proposition}\label{latticeClassification}
For an $\O_F$-lattice $\Lambda,$ $\mathrm{typ}(\Lambda)$ is independent of the choice of decomposition above, and gives us an injection
\begin{equation*}
    H\backslash \Lat(V)\xhookrightarrow{\mathrm{typ}}\Typ^0.
\end{equation*}
Moreover, as $(V,\langle\cdot,\cdot\rangle)$ varies, these maps cover the whole set $\Typ^0.$
\end{proposition}
\begin{proof}
For the case \Ugen, this follows from the work of Jacowitz \cite{Jacobowitz}. For the case \Ogen, this follows from the work of O'Meara \cite{Omeara}. For the case \SPgen, this is classical, see for example \cite[Section 2]{HironakaAlt}.
\end{proof}

\begin{theorem}\label{mainLemma}
Fix $\Lambda\in\Lat(V)$ and denote $\mathrm{typ}(\Lambda)=(e^0,\chi^0).$ Fix an orthogonal decomposition $V=\bigoplus_{i\in\Z}V^{(i)}_\Lambda$ as above, and denote $V^{(\ge j)}_\Lambda\defeq\bigoplus_{i\ge j}V^{(j)}_\Lambda.$

Consider another lattice $L\in\Lat(V)$ with $\varpi\Lambda\subseteq L\subseteq\Lambda,$  and denote $N_i\defeq\mathrm{proj}_{V^{(i)}_\Lambda}(L\cap V^{(\ge i)}_\Lambda).$ Then $\mathrm{typ}(L)$ is determined by $(N_i)_{i\in\Z}.$ More precisely: let $M_i\defeq N_i\cap \varpi^{i+1}N_i^\vee,$ and denote $n,m,l\colon\Z\to\Z_{\ge0}$ and $\psi_1,\psi_2\colon\Z\to\mathrm{Sign}$ to be
\begin{enumerate}
    \item $m(i)=\dim(M_i/\varpi\Lambda^{(i)}),$
    \item $(n(i),\psi_1(i))=\mathrm{typ}(N_i/M_i,\langle\cdot,\cdot\rangle^{(i)}),$
    \item $(l(i),\psi_2(i))=\mathrm{typ}(\varpi^{i+1}N_i^\vee/M_i,\langle\cdot,\cdot\rangle^{(i)}).$
\end{enumerate}
Note that these satisfy
\begin{equation*}
    e^0(i)=n(i)+\frac{2}{\cfactor}\cdot m(i)+l(i)\quad\text{and}\quad\chi^0(i)=\psi_1(i)\cdot\epsilon^{m(i)}\cdot\psi_2(i).
\end{equation*}
Then $\mathrm{typ}(L)=(f^0,\psi^0)$ is given by
\begin{equation*}
    f^0(i)=n(i)+\frac{2}{\cfactor}\cdot m(i-1)+l(i-2),\quad\text{and}\quad\psi^0(i)=\psi_1(i)\cdot\epsilon^{m(i-1)}\cdot\psi_2(i-2).\qedhere
\end{equation*}
\end{theorem}
\begin{proof}
First, on the $\F_q$ vector spaces $\Lambda^{(i)}/\varpi\Lambda^{(i)},$ we consider the following basis: i) choose a basis $\underline{u(i)}$ of $M_i/\varpi\Lambda^{(i)},$ ii) extend $\underline{u(i)}$ to a basis $\underline{u(i)},\underline{v(i)}$ of $N_i/\varpi\Lambda^{(i)},$ iii) extend $\underline{u(i)}$ to a basis $\underline{u(i)},\underline{x(i)}$ of $\varpi^{i+1}N_i^\vee/\varpi\Lambda^{(i)},$ iv) extend $\underline{u(i)},\underline{v(i)},\underline{x(i)}$ to a basis $\underline{u(i)},\underline{v(i)},\underline{x(i)},\underline{y_0(i)}$ of $\Lambda^{(i)}/\varpi\Lambda^{(i)}.$ We lift this to an $\O_F$-basis of $\Lambda^{(i)}$ and use the same notation for simplicity. Since $\Lambda^{(i)}$ is self-dual under $\langle\cdot,\cdot\rangle^{(i)},$ the moment matrix of $\Lambda^{(i)}$ must be in $\mathrm{GL}_{\gamma\cdot e^0(i)}(\O_F)$:
\begin{equation*}
    T_{\langle\cdot,\cdot\rangle^{(i)}}(\underline{u(i)},\underline{v(i)},\underline{x(i)},\underline{y_0(i)})=\begin{pmatrix}\varpi*&\varpi*&\varpi*&*\\\varpi*&T_{\langle\cdot,\cdot\rangle^{(i)}}(\underline{v(i)})&\varpi*&*\\ 
    \varpi*&\varpi*&T_{\langle\cdot,\cdot\rangle^{(i)}}(\underline{x(i)})&*\\
    *&*&*&T_{\langle\cdot,\cdot\rangle^{(i)}}(\underline{y_0(i)})\end{pmatrix}\in\mathrm{GL}_{\gamma\cdot e^0(i)}(\O_F),
\end{equation*}
where $*$ represents matrices with entries in $\O_F.$ A standard row reduction allow us to modify this basis in such a way that
\begin{equation*}
    T_{\langle\cdot,\cdot\rangle^{(i)}}(\underline{u(i)},\underline{v(i)},\underline{x(i)},\underline{y_0(i)})=\begin{pmatrix}0&0&0&A(i)\\0&T_{\langle\cdot,\cdot\rangle^{(i)}}(\underline{v(i)})&0&0\\ 
    0&0&T_{\langle\cdot,\cdot\rangle^{(i)}}(\underline{x(i)})&0\\
    A(i)^*&0&0&T_{\langle\cdot,\cdot\rangle^{(i)}}(\underline{y_0(i)})\end{pmatrix}.
\end{equation*}
Now we claim we may also assume $T_{\langle\cdot,\cdot\rangle^{(i)}}(\underline{y_0(i)})=0$: Consider $\underline{y(i)}=\underline{y_0(i)}+X\cdot\underline{u(i)}.$ Then
\begin{equation*}
    T_{\langle\cdot,\cdot\rangle^{(i)}}(\underline{y(i)})=T_{\langle\cdot,\cdot\rangle^{(i)}}(\underline{y_0(i)})+X\cdot A(i)+A(i)^*\cdot X^*.
\end{equation*}
We can choose $X$ so that this is $0$ as i) $A(i)$ is invertible and ii) the map $\mathrm{Mat}_{k\times k}(\O_F)\to\mathrm{Mat}_{k\times k}(\O_F)^{*=1}$ given by $X\mapsto X+X^*$ is surjective.\footnote{Here we are using that $q$ is odd in case \Ogen.}
We will write
\begin{equation*}
    T_{\langle\cdot,\cdot\rangle^{(i)}}(\underline{u(i)},\underline{v(i)},\underline{x(i)},\underline{y(i)})=\begin{pmatrix}0&0&0&M(i)\\0&V(i)&0&0\\ 
    0&0&X(i)&0\\
    M(i)^*&0&0&0\end{pmatrix}\in\mathrm{GL}_{\gamma\cdot e^0(i)}(\O_F).
\end{equation*}
Taking determinants, we have
\begin{equation*}
    \chi(i)=\mathrm{sign}\left((-1)^m\det(V(i)X(i)M(i)M(i)^*)\right)=\epsilon^m\cdot\psi_1(i)\cdot\psi_2(i).
\end{equation*}

Under this $\O_F$-basis
\begin{equation*}
    \ldots,\underline{u(i+1)},\underline{v(i+1)},\underline{x(i+1)},\underline{y(i+1)},\underline{u(i)},\underline{v(i)},\underline{x(i)},\underline{y(i)},\ldots
\end{equation*}
of $\Lambda,$ the lattice $L$ can be represented by the $\O_F$-span of the columns of a matrix of the form
\begin{equation}\label{Lmatrix}
    \begin{psmallmatrix}
        \ddots&&&&&&&&&&\\
        &I_{m(i)}&0&0&0&&0&0&0&0&\\
        &0&I_{\cfactor\cdot n(i)}&0&0&&0&0&0&0&\\
        &0&0&\varpi\cdot I_{\cfactor\cdot l(i)}&0&&*&*&0&0&\\
        &0&0&0&\varpi\cdot I_{m(i)}&&*&*&0&0&\\
        &&&&&\ddots&&&&&\\
        &0&0&0&0&&I_{m(j)}&0&0&0&\\
        &0&0&0&0&&0&I_{\cfactor\cdot n(j)}&0&0&\\
        &0&0&0&0&&0&0&\varpi\cdot I_{\cfactor\cdot l(j)}&0&\\
        &0&0&0&0&&0&0&0&\varpi\cdot I_{m(j)}&\\
        &&&&&&&&&&\ddots
    \end{psmallmatrix}
\end{equation}
where $i>j$ in the above. We will denote such basis of $L$ by
\begin{equation*}
\ldots,\underline{u_L(i+1)},\underline{v_L(i+1)},\underline{x_L(i+1)},\underline{y_L(i+1)},\underline{u_L(i)},\underline{v_L(i)},\underline{x_L(i)},\underline{y_L(i)},\ldots,
\end{equation*}
and we let $\underline{\alpha(i)}=(\underline{u_L(i)},\underline{v_L(i)},\underline{x_L(i)},\underline{y_L(i)}).$

Now we analyze the moment matrix for the $\O_F$-basis $(\underline{\alpha(i)})_{i\in\Z}$ of $L.$ We have
\begin{equation*}
\begin{split}
    \langle\underline{\alpha(i)},\underline{\alpha(i)}\rangle&=\begin{pmatrix}&&&\varpi^{i+1} M(i)\\&\varpi^i V(i)&&\\&&\varpi^{i+2} X(i)&\\\varpi^{i+1} M(i)^*&&&\end{pmatrix}+\sum_{k>i}\begin{pmatrix}\varpi^k*&\varpi^k*&0&0\\\varpi^k*&\varpi^k*&0&0\\0&0&0&0\\0&0&0&0\end{pmatrix}\\
    &=\begin{pmatrix}\varpi^{i+1}*&\varpi^{i+1}*&&\varpi^{i+1} M(i)\\\varpi^{i+1}*&\varpi^i V_L(i)&&\\&&\varpi^{i+2} X(i)&\\\varpi^{i+1} M(i)^*&&&\end{pmatrix}
\end{split}
\end{equation*}
for some $V_L(i)\equiv V(i)\mod\varpi\O_F,$ and if $i>j,$
\begin{equation*}
    \langle\underline{\alpha(i)},\underline{\alpha(j)}\rangle=\begin{pmatrix}\varpi^i*&\varpi^i*&0&0\\0&0&0&0\\\varpi^{i+1}*&\varpi^{i+1}*&0&0\\0&0&0&0\end{pmatrix}+\sum_{k>i}\begin{pmatrix}\varpi^k*&\varpi^k*&0&0\\\varpi^k*&\varpi^k*&0&0\\0&0&0&0\\0&0&0&0\end{pmatrix}=\begin{pmatrix}\varpi^i*&\varpi^i*&0&0\\\varpi^{i+1}*&\varpi^{i+1}&0&0\\\varpi^{i+1}*&\varpi^{i+1}*&0&0\\0&0&0&0\end{pmatrix}.
\end{equation*}

Now consider $\underline{\beta(i)}=(\underline{v_L(i)},\underline{x_L(i-2)},\underline{u_L(i-1)},\underline{y_L(i-1)}).$ Then
\begin{equation*}
\langle\underline{\beta(i)},\underline{\beta(i)}\rangle=\begin{pmatrix}\varpi^iV_L(i)&0&\varpi^{i+1}*&0\\0&\varpi^iX(i-2)&0&0\\\varpi^{i+1}*&0&\varpi^i*&\varpi^iM(i-1)\\0&0&\varpi^{i}M(i-1)^*&0\end{pmatrix},
\end{equation*}
as well as
\begin{equation*}
\langle\underline{\beta(i)},\underline{\beta(i-1)}\rangle=\begin{pmatrix}\varpi^{i+1}*&0&\varpi^{i+1}*&0\\0&0&0&0\\\varpi^{i}*&0&\varpi^{i-1}*&0\\0&0&0&0\end{pmatrix},\quad\langle\underline{\beta(i)},\underline{\beta(i-2)}\rangle=\begin{pmatrix}\varpi^{i+1}*&0&\varpi^{i+1}*&0\\0&0&\varpi^{i-1}*&0\\\varpi^{i-1}*&0&\varpi^{i-1}*&0\\0&0&0&0\end{pmatrix}.
\end{equation*}
and if $i>j+1,$
\begin{equation*}
\langle\underline{\beta(i)},\underline{\beta(j)}\rangle=\begin{pmatrix}\varpi^{i+1}*&0&\varpi^{i+1}*&0\\\varpi^{i-1}*&0&\varpi^{i-1}*&0\\\varpi^{i-1}*&0&\varpi^{i-1}*&0\\0&0&0&0\end{pmatrix}.
\end{equation*}
For each $i,$ since $\varpi^iM(i-1)$ is the only nonzero term in its columns, and since all the other terms to its left have valuation $\ge i,$ we can perform a series of row reductions to obtain a new basis $\ldots,\underline{\beta_0(i+1)},\underline{\beta_0(i)},\ldots$ such that
\begin{equation*}
\langle\underline{\beta_0(i)},\underline{\beta_0(i)}\rangle=\begin{pmatrix}\varpi^iV_L(i)&0&0&0\\0&\varpi^iX(i-2)&0&0\\0&0&\varpi^i*&\varpi^iM(i-1)\\0&0&\varpi^{i}M(i-1)^*&0\end{pmatrix}\eqdef E(i),
\end{equation*}
and such that
\begin{equation*}
\langle\underline{\beta_0(i)},\underline{\beta_0(j)}\rangle=\begin{pmatrix}\varpi^{i+1}*&0&0&0\\0&0&0&0\\0&0&0&0\\0&0&0&0\end{pmatrix}\quad\text{for }i>j\ge i-2,
\end{equation*}
and
\begin{equation*}
    \langle\underline{\beta_0(i)},\underline{\beta_0(j)}\rangle=\begin{pmatrix}\varpi^{i+1}*&0&0&0\\\varpi^{i-1}*&0&0&0\\0&0&0&0\\0&0&0&0\end{pmatrix}\quad \text{for } i-2>j.
\end{equation*}
Finally, it is easy to see that there is a further row reduction which makes the moment matrix become block-diagonal with blocks $E(i).$ Now we have $\frac{1}{\varpi^i}E(i)\in\mathrm{GL}_{\cfactor\cdot n(i)+2\cdot m(i-1)+\cfactor\cdot l(i-2)}(\O_F)$ and
\begin{equation*}
    \mathrm{sign}(\det(E(i)/\varpi^i))=\mathrm{sign}(\det V_L(i))\cdot\mathrm{sign}(\det X(i-2))\cdot\epsilon^{m(i-1)}=\psi_1(i)\cdot\psi_2(i-2)\cdot\epsilon^{m(i-1)},
\end{equation*}
and thus $f^0(i)=n(i)+\frac{2}{\cfactor}\cdot m(i-1)+l(i-2)$ and $\psi^0(i)=\psi_1(i)\cdot\epsilon^{m(i-1)}\cdot\psi_2(i-2),$ as claimed.
\end{proof}

\begin{corollary}\label{mainLemmaCount}
Fix $\Lambda\in\Lat(V)$ and denote $\mathrm{typ}(\Lambda)=(e^0,\chi^0).$ Consider another pair $(f^0,\psi^0).$ If there exist $L\in\Lat(V)$ with $\varpi\Lambda\subseteq L\subseteq\Lambda$ and $\mathrm{typ}(L)=(f^0,\psi^0),$ then there exist $n,m,l\colon\Z\to\Z_{\ge0}$ and $\psi_1,\psi_2\colon\Z\to\mathrm{Sign}$ with
\begin{equation*}
    e^0(i)=n(i)+\frac{2}{\cfactor}\cdot m(i)+l(i)\quad\chi^0(i)=\psi_1(i)\cdot\epsilon^{m(i)}\cdot\psi_2(i)
\end{equation*}
and
\begin{equation*}
    f^0(i)=n(i)+\frac{2}{\cfactor}\cdot m(i-1)+l(i-2),\quad \psi^0(i)=\psi_1(i)\cdot\epsilon^{m(i-1)}\cdot\psi_2(i-2).
\end{equation*}
In this case, the number of such $L$ is
\begin{equation*}
    \sum_{(n,m,l,\psi_1,\psi_2)}q^{\sum_{i>j}(m(i)+\cfactor\cdot l(i))(m(j)+\cfactor\cdot n(j))}\prod_iQ\Typec{n(i)}{\psi_1(i)}{m(i)}{}{l(i)}{\psi_2(i)},
\end{equation*}
where the sum is over all tuples $(n,m,l,\psi_1,\psi_2)$ as above.
\end{corollary}
\begin{proof}
This follows from \Cref{mainLemma}: for each $i,$ given $(n(i),\psi_1(i),m(i),l(i),\psi_2(i)),$ the number of $N_i$ is $Q\Typec{n(i)}{\psi_1(i)}{m(i)}{}{l(i)}{\psi_2(i)}$ by \Cref{mainLemmaFinite}; given all the $N_i,$ the number of $L$ is
\begin{equation*}
    q^{\sum_{i>j}(m(i)+\cfactor\cdot l(i))(m(j)+\cfactor\cdot n(j))}
\end{equation*}
by the description in \eqref{Lmatrix}.
\end{proof}

%% file: 4-Straightening.tex
\section{Explicit models of spherical functions}\label{StraighteningSection}
We fix a coefficient ring $R\neq0.$\index{R@$R$} In the case \Ogen, we assume $2\in R^\times.$
\subsection{Spherical functions and Hecke action}
Let $r\ge1$ be a positive integer, and let $V=F^{\cfactor r},$ that is,
\begin{case}{\Ugen}
$V=F^r,$
\end{case}
\begin{case}{\Ogen}
$V=F^r,$
\end{case}
\begin{case}{\SPgen}
$V=F^{2r}.$
\end{case}
Let $G\defeq\mathrm{GL}_{\cfactor r}(F)=\mathrm{GL}(V)$ be the associated $p$-adic Lie group, with maximal open compact subgroup $K\defeq\mathrm{GL}_{\cfactor r}(\O_F).$ We consider
\index{X@$X$}
\begin{equation*}
    X\defeq (G)^{*=1}=\begin{cases}\mathrm{Herm}_r(F/F_0)\cap G&\text{in the case \Ugen,}\\\mathrm{Sym}_r(F)\cap G&\text{in the case \Ogen,}\\\mathrm{Alt}_{2r}(F)\cap G&\text{in the case \SPgen}\end{cases}
\end{equation*}
the space of nondegenerate Hermitian/symmetric/alternating matrices.

The group $G$ acts on $X$ on the right by $x\cdot g\defeq g^* xg.$ This allow us to identify
\begin{equation*}
    X=\bigsqcup_{(V,\langle\cdot,\cdot\rangle)} H(V,\langle\cdot,\cdot,\rangle)\backslash G,
\end{equation*}
where $\langle\cdot,\cdot\rangle$ runs through the (finitely many) isomorphism classes of nondegenerate
\begin{case}{\Ugen}
$F/F_0$-Hermitian forms on $V,$
\end{case}
\begin{case}{\Ogen}
symmetric forms on $V,$
\end{case}
\begin{case}{\SPgen}
alternating forms on $V,$
\end{case}
and $H(V,\langle\cdot,\cdot\rangle)$ denotes the corresponding automorphism groups.

\begin{definition}
We denote $\Typ^0_r$\index{Typ0r@$\Typ^0_r$} to be the subset of $\Typ^0$ consisting of $(e^0,\chi^0)$ with $\sum_ie^0(i)=r.$
\end{definition}

\begin{proposition}\label{XoverKclassification}
We have a bijection
\begin{equation*}
    X/K=\bigsqcup_{(V,\langle\cdot,\cdot,\rangle)}H(V,\langle\cdot,\cdot,\rangle)\backslash G/K\xrightiso{\mathrm{typ}}\Typ^0_r.
\end{equation*}
\end{proposition}
\begin{proof}
Note that $G/K\iso\Lat(V)$ via $g\mapsto g\cdot(\O_F^{\cfactor r}).$ With this, the claim is a re-statement of \Cref{latticeClassification}.
\end{proof}

\begin{definition}
We denote $C^\infty(X/K)$
\index{C(X/K)@$C^\infty(X/K)$}
to be the $R$-module of locally constant functions $X/K\to R.$ We denote $\mathcal{S}(X/K)\defeq C^\infty_c(X/K)\subseteq C^\infty(X/K)$\index{S(X/K)@$\mathcal{S}(X/K)$} to be the submodule of compactly supported functions. These are modules over the ($R$-valued) Hecke algebra $\mathcal{H}(G,K)\defeq C_c^\infty(K\backslash G\slash K)$\index{H(G,K)@$\mathcal{H}(G,K)$} via convolution.
\end{definition}
\Cref{XoverKclassification} induces isomorphisms
\begin{equation*}
    C^\infty(X/K)\rightiso R\llbracket\Typ^0_r\rrbracket,\quad\mathcal{S}(X/K)\rightiso R[\Typ^0_r],
\end{equation*}
and we use these to transport the action of $\mathcal{H}(G,K)$ to $R[\Typ^0_r]$ and $R\llbracket\Typ^0_r\rrbracket.$ For $T\in\mathcal{H}(G,K),$ we will denote $T^*\colon R[\Typ_r^0]\to R[\Typ_r^0]$\index{0()*@$(\cdot)^*$} to be the adjoint of $T$ under the perfect pairing
\index{0\{,\}@$\{\cdot,\cdot\}$}
\begin{equation*}
    \pair{\cdot}{\cdot}\colon R\llbracket\Typ^0_r\rrbracket\times R[\Typ^0_r]\to R.
\end{equation*}
\begin{definition}
For $0\le k\le \cfactor r,$ we denote $\mu_k\defeq(\underbrace{1,\ldots,1}_k,\underbrace{0,\ldots,0}_{\cfactor r-k})$ and $T_{k,r}\defeq\mathrm{char}(K\cdot\varpi^{\mu_k}\cdot K)$\index{Tk,r@$T_{k,r}$} the miniscule Hecke operators.
\end{definition}

Concretely, the action $T_{k,r}^*\colon R[\Typ_r^0]\to R[\Typ_r^0]$ is encoded as follows: if $(e^0,\chi^0)\in\Typ^0_r$ and $\Lambda\in\Lat(V)$ is such that $\mathrm{typ}(\Lambda)=(e^0,\chi^0),$ then
\begin{equation*}
    T_{k,r}^*(e^0,\chi^0)=\sum_{\substack{\varpi\Lambda\subseteq L\subseteq\Lambda\\\mathrm{length}(\Lambda/L)=k}}1\cdot\mathrm{typ}(L).
\end{equation*}
The right hand side is, of course, computed in \Cref{mainLemmaCount}. We will now work to show that it can be succinctly described by certain straightening relations.

\subsection{Straightening relations}
\begin{definition}\label{Typdef}
    For each $r\ge0,$ we consider $\Typ_r\defeq(\Z\times\mathrm{Sign})^r,$ and we denote $\Typ\defeq\bigsqcup_{r\ge0}\Typ_r.$\index{Typ@$\Typ,\ \Typ_r$} For $a,b\ge0,$ we consider the concatenation product $\star\colon\Typ_a\times\Typ_b\to\Typ_{a+b}.$\index{0*@$\star$} This makes $R[\Typ]$ into a graded noncommutative $R$-ring with graded pieces $\mathrm{Gr}_r(R[\Typ])=R[\Typ_r].$
\end{definition}

For $e\in\Z^r,$ $\chi\in\mathrm{Sign}^r,$ we will denote the element associated to $(e,\chi)\in\Typ_r$ by
\index{0(e_i,chi_i)@$\Typee{e_i}{\chi_i}{e_{i+1}}{\chi_{i+1}}$}\index{1delta(e,chi)@$\delta(e,\chi)$}
\begin{equation*}
    \delta(e,\chi)=\Typee{e_i}{\chi_i}{e_{i+1}}{\chi_{i+1}}\in R[\Typ].
\end{equation*}
In the case \Ogen, we will also need a certain variation: for a character $s\colon\mathrm{Sign}^r\to\{\pm1\},$ we denote
\index{1deltas(e)@$\delta_s(e)$}
\begin{equation*}
    \delta_s(e)\defeq\sum_{\chi\in\mathrm{Sign}^r}s(\chi)\Typee{e_i}{\chi_i}{e_{i+1}}{\chi_{i+1}}\in R[\Typ].
\end{equation*}
Note that we are assuming that $\#\mathrm{Sign}\in R^\times,$ and so we may write
\begin{equation*}
    \Typee{e_i}{\chi_i}{e_{i+1}}{\chi_{i+1}}=\frac{1}{(\#\mathrm{Sign})^r}\sum_{s\colon\mathrm{Sign}^r\to\{\pm1\}}s(\chi)\delta_s(e).
\end{equation*}

Moreover, if $(s_1,\ldots,s_r)\in(\widehat{\mathrm{Sign}})^r=\mathrm{Sign}^r$ are such that $s\colon\mathrm{Sign}^r\to\{\pm1\}$ is given by $s(\chi)=\prod_{i}s_i(\chi_i),$ then we also denote
\index{0(e_i,s_i)Sigma@$\Typee{e_i}{s_i}{e_{i+1}}{s_{i+1}}^\Sigma$}
\begin{equation*}
    \Typee{e_i}{s_i}{e_{i+1}}{s_{i+1}}^\Sigma=\delta_s(e).
\end{equation*}
Observe that $\delta_{s}(e)\star\delta_{s'}(e')$ is given simply by the concatenation
\begin{equation*}
    \Typee{e_a}{s_a}{e_1'}{s_1'}^\Sigma.
\end{equation*}

\begin{definition}\label{Reldef}
We consider the homogeneous two-sided ideal $\Rel\subseteq R[\Typ]$ which is generated by the following degree $2$ elements. We will also denote $\Rel_r\defeq\mathrm{Gr}_r(\Rel)\subseteq R[\Typ_r].$\index{Rel@$\Rel,\ \Rel_r$}
\begin{case}{\Ugen}
1) For $a\in\Z,$
\begin{equation*}
\Rel(a)\defeq\typeC{a}{a+1}-\typeC{a+1}{a}.
\end{equation*}

2) For $b>a,$
\begin{equation*}
\Rel(a,b)\defeq\typeC{a}{b}-\typeC{a+1}{b-1}-(-q_0)^{b-a-1}\left(\typeC{b}{a}-\typeC{b-1}{a+1}\right).
\end{equation*}
\end{case}
\begin{case}{\Ogen}
1) For $b-a\in2\Z_{\ge0}+1,$
\begin{equation*}
\begin{split}
\Rel\TypeC{a}{s_1}{b}{s_2}\defeq\ &
    \TypeC{a}{s_1}{b}{s_2}^\Sigma-\TypeC{a+1}{s_2}{b-1}{s_1}^\Sigma\\
    &-q^{(b-a-1)/2}\left(\TypeC{b}{s_2}{a}{s_1}^\Sigma-\TypeC{b-1}{s_1}{a+1}{s_2}^\Sigma\right).
\end{split}
\end{equation*}
2) For $b-a\in2\Z_{\ge0}$
\begin{equation*}
    \Rel\TypeC{a}{s}{b}{-s}\defeq \TypeC{a}{s}{b}{-s}^\Sigma+q^{(b-a)/2}\TypeC{b}{s}{a}{-s}^\Sigma
\end{equation*}
and if we denote
\begin{equation*}
\begin{split}
    \mathrm{RelHalf}\TypeC{a}{s}{b}{s}\defeq\ &-\TypeC{a}{s}{b}{s}^\Sigma+\TypeC{a-1}{s}{b+1}{s}^\Sigma\\
    &-\epsilon\left(\TypeC{a}{-s}{b}{-s}^\Sigma-\TypeC{a+1}{-s}{b-1}{-s}^\Sigma\right),
\end{split}
\end{equation*}
then also, for $b-a\in2\Z_{\ge1},$
\begin{equation*}
    \Rel\TypeC{a}{s}{b}{s}\defeq \mathrm{RelHalf}\TypeC{a+1}{s}{b-1}{s}+q^{(b-a-2)/2}\mathrm{RelHalf}\TypeC{b-1}{s}{a+1}{s}.
\end{equation*}
\end{case}
\begin{case}{\SPgen}
1) For $a\in\Z,$
\begin{equation*}
\Rel(a)\defeq\typeC{a}{a+1}-\typeC{a+1}{a}.
\end{equation*}

2) For $b>a,$
\begin{equation*}
\Rel(a,b)\defeq\typeC{a}{b}-\typeC{a+1}{b-1}-q^{2(b-a-1)}\left(\typeC{b}{a}-\typeC{b-1}{a+1}\right)
\end{equation*}
\end{case}
\end{definition}
\begin{remark}
Note that in cases \Ugen and \SPgen, $\Rel(a,a+1)=2\Rel(a),$ so $\Rel(a)$ is not necessary if $2\in R^\times.$
\end{remark}
\begin{remark}\label{strimplications}
In case \Ogen, we have
\begin{equation*}
    \Rel\TypeC{a}{s}{a}{-s}=2\TypeC{a}{s}{a}{-s}^\Sigma,
\end{equation*}
and applying this for both $s=+$ and $s=-$ give us
\begin{equation}\label{changeDets}
    \TypeC{a}{\psi_1}{a}{\psi_2}\equiv\TypeC{a}{-\psi_1}{a}{-\psi_2}\mod\Rel.
\end{equation}

Moreover,
\begin{equation*}
    \Rel\TypeC{a}{s_1}{a+1}{s_2}=2\left(\TypeC{a}{s_1}{a+1}{s_2}^\Sigma-\TypeC{a+1}{s_2}{a}{s_1}^\Sigma\right)
\end{equation*}
and since we are assuming $2$ is invertible, this implies
\begin{equation}\label{relation1inScase}
    \TypeC{a}{\psi_1}{a+1}{\psi_2}\equiv\TypeC{a+1}{\psi_2}{a}{\psi_1}\mod \Rel.
\end{equation}
Similarly, $\Rel\TypeC{a}{s}{a+2}{s}=2\mathrm{RelHalf}\TypeC{a+1}{s}{a+1}{s},$ so we also have
\begin{equation*}
    \TypeC{a}{s}{a+2}{s}^\Sigma\equiv\TypeC{a+1}{s}{a+1}{s}^\Sigma+\epsilon\TypeC{a+1}{-s}{a+1}{-s}^\Sigma-\epsilon\TypeC{a+2}{-s}{a}{-s}^\Sigma\mod\Rel
\end{equation*}
and together with $\Rel\TypeC{a}{s}{a+2}{-s},$ this implies
\begin{equation}\label{concreteRel}
\begin{split}
    \TypeC{a}{\psi_1}{a+2}{\psi_2}\equiv\ &(1+\epsilon\psi_1\psi_2)\TypeC{a+1}{\psi_1}{a+1}{\psi_2}\\&-\epsilon\psi_1\psi_2\left(\frac{1+q\epsilon}{2}\TypeC{a+2}{\psi_2}{a}{\psi_1}+\frac{1-q\epsilon}{2}\TypeC{a+2}{-\psi_2}{a}{-\psi_1}\right).
\end{split}
\end{equation}
\end{remark}

For $a\in\Z$ and $\psi\in\mathrm{Sign},$ we will often write
\index{0(f^k,psi)@$\Type{(a)^k}{\psi}$}
\begin{equation*}
    \TypeCC{(a)^k}{\psi}\in R[\Typ_k]/\Rel_k
\end{equation*}
to denote the image of any concatenation product of the form $\TypeCC{a}{\psi_1}\star\cdots\star\TypeCC{a}{\psi_k}$ with $\prod\psi_i=\psi.$ Note that this is well defined by \eqref{changeDets}.

\begin{definition}\label{tdef}
For $\varepsilon\in\Z^r,$ we consider the translation endomorphisms $t(\varepsilon)\colon R[\Typ_r]\to R[\Typ_r]$\index{t(epsilon)@$t(\varepsilon)$} given by $\delta(e,\chi)\mapsto\delta(e+\varepsilon,\chi).$ We also denote $t_i(k)\defeq t(0,\ldots,0,k,0,\ldots,0)$ where $k\in\Z$ is in the $i$-th position.
\end{definition}

\begin{proposition}\label{preserveshalf}
    The following operators preserve $\Rel_2.$
    \begin{equation*}
        \begin{array}{c|c|c}
        \Ugen&\Ogen&\SPgen\\
        \hline
        -q_0\cdot t_1(1)+t_2(1)&q\cdot t_1(2)+t_2(2)&q^2\cdot t_1(1)+t_2(1)
        \end{array}
    \end{equation*}
\end{proposition}
\begin{proof}
\begin{case}{\Ugen}
This is because
\begin{equation*}
\begin{split}
    (-q_0\cdot t_1(1)+t_{2}(1))(\Rel(a))&=\Rel(a,a+2),\\
    (-q_0\cdot t_1(1)+t_{2}(1))(\Rel(a,b))&=(-q_0)\Rel(a+1,b)+\Rel(a,b+1)\quad\text{for } b>a+1.
\end{split}
\end{equation*}
\end{case}
\begin{case}{\Ogen}
This is because, for most $\Rel\TypeC{a}{s_1}{b}{s_2},$ the following expression is valid
\begin{equation*}
    (q\cdot t_1(2)+t_2(2))\Rel\TypeC{a}{s_1}{b}{s_2}=q\Rel\TypeC{a+2}{s_1}{b}{s_2}+\Rel\TypeC{a}{s_1}{b+2}{s_2}.
\end{equation*}
The exceptions are:
\begin{equation*}
\begin{split}
    (q\cdot t_1(2)+t_2(2))\Rel\TypeC{a}{s}{a}{-s}&=2\Rel\TypeC{a}{s}{a+2}{-s},\\
    (q\cdot t_1(2)+t_2(2))\Rel\TypeC{a}{s_1}{a+1}{s_2}&=2\Rel\TypeC{a}{s_1}{a+3}{s_2},\\
    (q\cdot t_1(2)+t_2(2))\Rel\TypeC{a}{s}{a+2}{s}&=2\Rel\TypeC{a}{s}{a+4}{s}.
\end{split}
\end{equation*}
\end{case}
\begin{case}{\SPgen}
This is because
\begin{equation*}
\begin{split}
    (q^2\cdot t_1(1)+t_{2}(1))(\Rel(a))&=\Rel(a,a+2),\\
    (q^2\cdot t_1(1)+t_{2}(1))(\Rel(a,b))&=q^2\Rel(a+1,b)+\Rel(a,b+1)\quad\text{for } b>a+1.\qedhere
\end{split}
\end{equation*}
\end{case}
\end{proof}

\begin{proposition}\label{BergmanProp}
    We have a canonical isomorphism $\mathrm{str}\colon R[\Typ_r]/\Rel_r\rightiso R[\Typ^0_r]$\index{str@$\mathrm{str}$} characterized by the fact that if $\delta(e,\chi)\in\Typ_r$ is with $e_1\ge\cdots\ge e_r,$ then $\mathrm{str}(\delta(e,\chi))=(e^0,\chi^0),$ where $e^0(i)=\lambda_i(e)$ and $\chi^0(i)=\prod_{k\colon e_k=i}\chi_k.$
\end{proposition}
\begin{proof}
This is an application of Bergman's diamond lemma \cite[Theorem 1.2]{Bergman}. While there are a priori infinitely many ambiguities to be checked, we will use \Cref{preserveshalf} to reduce them to a finite number.

In order to apply the lemma, we consider $\TypeCC{a}{s}^\Sigma$ for $a\in\Z$ and $s\in\widehat{\mathrm{Sign}}$ as our set of generators. We endow them with the partial order $\TypeCC{a}{s_1}\le\TypeCC{b}{s_2}$ if either $a<b$ or $(a,s_1)=(b,s_2).$ We extend this lexicographically to a partial order on the monomials.\footnote{Strictly speaking, we must make two remarks: i) We are applying Bergman's diamond lemma with the reverse ordering as the one we defined. ii) Our ordering does not satisfy the ascending chain condition, but for any given monomial in $R[\Typ],$ there is a bound on the size of the monomials that can show up in the straightening process. So the lemma still applies.}

Let
\begin{equation*}
    J_2\defeq\left\{\TypeC{a}{s_1}{b}{s_2}\colon \TypeCC{a}{s_1}\not\ge\TypeCC{b}{s_2}\right\},\quad J_3\defeq\left\{\Typec{a}{s_1}{b}{s_2}{c}{s_3}\colon \TypeC{a}{s_1}{b}{s_2},\TypeC{b}{s_2}{c}{s_3}\in J_2\right\}.
\end{equation*}
We note that we have a set of generators of $\Rel$ indexed by $J_2$ of the form
\begin{equation*}
    r\TypeC{a}{s_1}{b}{s_2}=\TypeC{a}{s_1}{b}{s_2}^\Sigma+\left(\text{lexicographically larger terms}\right).
\end{equation*}

Namely, for $\TypeC{a}{s_1}{b}{s_2}\in J_2$ we consider
\begin{itemize}[leftmargin=*]
    \item $r(a,a+1)=\frac{1}{2}\Rel(a,a+1)=\Rel(a)$ in cases \Ugen and \SPgen,
    \item $r\TypeC{a}{s_1}{a+1}{s_2}=\frac{1}{2}\Rel\TypeC{a}{s_1}{a+1}{s_2}=\TypeC{a}{s_1}{a+1}{s_2}^\Sigma-\TypeC{a+1}{s_2}{a}{s_1}^\Sigma$ in case \Ogen,
    \item $r\TypeC{a}{s}{a+2}{s}=\frac{1}{2}\Rel\TypeC{a}{s}{a+2}{s}=\mathrm{RelHalf}\TypeC{a+1}{s}{a+1}{s}$ in case \Ogen,
    \item $r\TypeC{a}{s}{a}{-s}=\TypeC{a}{s}{a}{-s}^\Sigma$ in case \Ogen.
\end{itemize}
and $r\TypeC{a}{s_1}{b}{s_2}=\Rel\TypeC{a}{s_1}{b}{s_2}$ in the remaining cases. We extend the notation to $r\TypeC{a}{s_1}{b}{s_2}=0$ if $\TypeCC{a}{s_1}\ge\TypeCC{b}{s_2}.$

Denote
\begin{equation*}
    I_2\TypeC{a}{s_1}{b}{s_2}\defeq R\left[r\TypeC{a'}{s_1'}{b'}{s_2'}\colon \TypeC{a'}{s_1'}{b'}{s_2'}>\TypeC{a}{s_1}{b}{s_2}\right].
\end{equation*}
Similarly, denote
\begin{equation*}
\begin{split}
    I_3\Typec{a}{s_1}{b}{s_2}{c}{s_3}\defeq R\left[\TypeCC{a'}{s_1'}^\Sigma\right.\star r\TypeC{b'}{s_2'}{c'}{s_3'},&\quad r\TypeC{a'}{s_1'}{b'}{s_2'}\star\TypeCC{c'}{s_3'}^\Sigma\colon\\
    &\qquad\qquad\left.\vphantom{\TypeCC{a'}{s_1'}^\Sigma}\Typec{a'}{s_1'}{b'}{s_2'}{c'}{s_3'}>\Typec{a}{s_1}{b}{s_2}{c}{s_3}\right].
\end{split}
\end{equation*}
In order to apply Bergman's lemma, we need to check for all $\Typec{a}{s_1}{b}{s_2}{c}{s_3}\in J_3$ that
\begin{equation*}
    D\Typec{a}{s_1}{b}{s_2}{c}{s_3}\defeq\TypeCC{a}{s_1}^\Sigma\star r\TypeC{b}{s_2}{c}{s_3}-r\TypeC{a}{s_1}{b}{s_2}\star\TypeCC{c}{s_3}^\Sigma
\end{equation*}
is in $I_3\Typec{a}{s_1}{b}{s_2}{c}{s_3}.$ We denote this statement as
\begin{equation*}
    P\Typec{a}{s_1}{b}{s_2}{c}{s_3}\colon\quad\left(D\Typec{a}{s_1}{b}{s_2}{c}{s_3}\in I_3\Typec{a}{s_1}{b}{s_2}{c}{s_3}\right).
\end{equation*}

    Denote $\beta\defeq\begin{cases}
        1 & \text{in cases }\Ugen,\SPgen, \\
        2& \text{in case }\Ogen.
    \end{cases}$ We will prove the following implications
    \begin{equation}\label{Bergman1}
        P\Typec{a}{s_1}{b}{s_2}{c-\beta}{s_3}\implies P\Typec{a}{s_1}{b}{s_2}{c}{s_3}\quad\text{for }\Typec{a}{s_1}{b}{s_2}{c-\beta}{s_3}\in J_3,
    \end{equation}
    and
    \begin{equation}\label{Bergman2}
        P\Typec{a+\beta}{s_1}{b}{s_2}{c}{s_3}\implies P\Typec{a}{s_1}{b}{s_2}{c}{s_3}\quad\text{for }\Typec{a+\beta}{s_1}{b}{s_2}{c}{s_3}\in J_3.
    \end{equation}
    Together with the fact that $P\Typec{a}{s_1}{b}{s_2}{c}{s_3}\iff P\Typec{a+k}{s_1}{b+k}{s_2}{c+k}{s_3}$ for any $k\in\Z,$ this would reduce the claim to checking a finite number of cases.

    We consider the following two operators $\varphi_2^+$ and $\varphi_2^-.$
    \begin{equation*}
    \begin{array}{c|c|c}
        &\varphi_2^+&\varphi_2^-\\
        \hline
        \SPgen&q^2\cdot t_1(1)+t_2(1)&q^2\cdot t_2(-1)+t_1(-1)\\
        \hline
        \Ugen&-q_0\cdot t_1(1)+t_2(1)&-q_0\cdot t_2(-1)+t_1(-1)\\
        \hline
        \Ogen&q\cdot t_1(2)+t_2(2)&q\cdot t_2(-2)+t_1(-2)
    \end{array}
    \end{equation*}
    By the (proof of) \Cref{preserveshalf}, we have, for $\TypeC{a}{s_1}{b}{s_2}\in J_2,$ that
    \begin{equation*}
    \begin{split}
        \varphi_2^+\left(r\TypeC{a}{s_1}{b}{s_2}\right)&\equiv r\TypeC{a}{s_1}{b+\beta}{s_2}\mod I_2\TypeC{a}{s_1}{b+\beta}{s_2},\\
        \varphi_2^-\left(r\TypeC{a}{s_1}{b}{s_2}\right)&\equiv r\TypeC{a-\beta}{s_1}{b}{s_2}\mod I_2\TypeC{a-\beta}{s_1}{b}{s_2}.
    \end{split}
    \end{equation*}
    In particular, if we consider the operators $\varphi_3^+$ and $\varphi_3^-$ given by
    \begin{equation*}
    \begin{array}{c|c|c}
        &\varphi_3^+&\varphi_3^-\\
        \hline
        \SPgen&q^4\cdot t_1(1)+q^2\cdot t_2(1)+t_3(1)&q^4\cdot t_3(-1)+q^2\cdot t_2(-1)+t_1(-1)\\
        \hline
        \Ugen&q_0^2\cdot t_1(1)-q_0\cdot t_2(1)+t_3(1)&q_0^2\cdot t_3(-1)-q_0\cdot t_2(-1)+t_1(-1),\\
        \hline
        \Ogen&q^2\cdot t_1(2)+q\cdot t_2(2)+t_3(2)&q^2\cdot t_3(-2)+q\cdot t_2(-2)+t_1(-2)
    \end{array}
    \end{equation*}
    then we have, for $\Typec{a}{s_1}{b}{s_2}{c}{s_3}\in J_3,$ that
    \begin{equation*}
    \begin{split}
        \varphi_3^+\left(D\Typec{a}{s_1}{b}{s_2}{c}{s_3}\right)&\equiv D\Typec{a}{s_1}{b}{s_2}{c+\beta}{s_3}\mod I_3\Typec{a}{s_1}{b}{s_2}{c+\beta}{s_3},\\
        \varphi_3^-\left(D\Typec{a}{s_1}{b}{s_2}{c}{s_3}\right)&\equiv D\Typec{a-\beta}{s_1}{b}{s_2}{c}{s_3}\mod I_3\Typec{a-\beta}{s_1}{b}{s_2}{c}{s_3}.
    \end{split}
    \end{equation*}
    This give us \eqref{Bergman1} and \eqref{Bergman2}.

    In cases \SPgen and \Ugen, we are reduced to consider $P(0,1,2).$ For the case \Ogen, and also exploiting the order-preserving symmetry $\Typec{a}{s_1}{b}{s_2}{c}{s_3}\mapsto \Typec{-c}{s_3}{-b}{s_2}{-a}{s_1},$ we are reduced to consider
    \begin{equation*}
    \begin{split}
       &P\Typec{0}{-s}{0}{s}{0}{-s},\quad P\Typec{0}{-s_1}{0}{s_1}{1}{s_2},\quad P\Typec{0}{-s}{0}{s}{2}{s},\\
       &P\Typec{0}{s_1}{1}{s_2}{2}{s_3},\quad P\Typec{0}{s_1}{1}{s_2}{3}{s_2},\quad P\Typec{0}{s}{2}{s}{4}{s}.
    \end{split}
    \end{equation*}
    This is a finite computation which can be easily verified on a computer.\footnote{See \href{https://github.com/murilocorato/Straightening-relations}{https://github.com/murilocorato/Straightening-relations} for an implementation of this verification on Sage.}
\end{proof}
\begin{remark}\label{basisremark}
Note that, under the above, we have the following basis of $R[\Typ_r^0]$:\footnote{Note that if $e_i=e_j$ but $s_i\neq s_j,$ then we automatically have $\delta_s(e)\in\Rel$ because of \eqref{changeDets}.}
\begin{equation*}
    \left\{\delta_s(e)\colon e\in\Typ_r,\ s\colon\mathrm{Sign}^r\to\{\pm1\}\text{ s.t. }e_1\ge\cdots\ge e_r\text{ and }e_i=e_j\implies s_i=s_j\right\}.
\end{equation*}
 Moreover, this set is orthonormal under $(\#\mathrm{Sign})^{-r}\pair{\cdot}{\cdot}.$
\end{remark}

\begin{definition}\label{Deltadefinition}
Consider $\Delta_r(x)\colon R[\Typ_r]\to R[x][\Typ_r]$ given by
\begin{equation*}
    \Delta_r(x)\defeq\begin{cases}\displaystyle
    \sum_{\varepsilon\in\{0,1\}^r}q^{\inv(\varepsilon)}\cdot x^{\lambda_1(\varepsilon)}\cdot t(2\varepsilon)&\text{in the case \Ugen,}\\
    \displaystyle\sum_{\varepsilon\in\{0,1\}^r}q^{\inv(\varepsilon)}\cdot x^{\lambda_1(\varepsilon)}\cdot t(2\varepsilon)&\text{in the case \Ogen,}\\
    \displaystyle\sum_{\varepsilon\in\{0,1,2\}^r}q^{2\cdot\widetilde{\inv}(\varepsilon)}(q+1)^{\lambda_1(\varepsilon)}q^{\binom{\lambda_1(\varepsilon)}{2}}\cdot x^{\Sigma(\varepsilon)}\cdot t(\varepsilon)&\text{in the case \SPgen.}\end{cases}
\end{equation*}
We also let $\Delta_{k,r}\colon R[\Typ_r]\to R[\Typ_r]$\index{1Deltak,r@$\Delta_{k,r},\ \Delta_r(x)$} be $\Delta_{k,r}\defeq [x^k](\Delta_r(x)).$
\end{definition}
\begin{remark}
For the case \SPgen, we may also write $\Delta_r(x)$ as
\begin{equation*}
    \Delta_r(x)=\sum_{\varepsilon\in\{0,1\}^{2r}}q^{\inv(\varepsilon)}x^{\lambda_1(\varepsilon)}\cdot t(\varepsilon_1+\varepsilon_2,\ldots,\varepsilon_{2r-1}+\varepsilon_{2r}).
\end{equation*}
\end{remark}
\begin{proposition}\label{Deltapreserves}
$\Delta_{k,r}$ preserves $\Rel_r.$
\end{proposition}
\begin{proof}
This follows from \Cref{preserveshalf}, the fact that $t(1,1)$ preserves $\Rel_2,$ and the following observations: in case \Ugen,
\begin{equation*}
    q\cdot t_1(2)+t_{2}(2)=(-q_0\cdot t_1(1)+t_2(1))^2+2q_0\cdot t(1,1),
\end{equation*}
and in case \SPgen,
\begin{equation*}
    q^4\cdot t_1(2)+t_{2}(2)=(q^2\cdot t_1(1)+t_2(1))^2-2q^2\cdot t(1,1).\qedhere
\end{equation*}
\end{proof}
\begin{theorem}\label{inducedThm}
The induced map $\mathrm{str}^*\Delta_{k,r}\colon R[\Typ^0_r]\to R[\Typ^0_r]$ agrees with $T_{k,r}^*.$
\end{theorem}
\begin{proof}
Given $(e^0,\chi^0)\in\Typ^0_r,$ choose $(e,\chi)\in\Typ_r$ with $\mathrm{str}(\delta(e,\chi))=(e^0,\chi^0).$ We will prove that $\mathrm{str}(\Delta_{k,r}(\delta(e,\chi)))=T_{k,r}^*(e^0,\chi^0)$ in two steps. First, we reduce this to the case $e=(0,\ldots,0).$ Second, we prove this case by induction on $r.$

As in \Cref{mainLemmaCount}, given the data $D=(n,m,l,\psi_1,\psi_2),$ we denote $(f^0(D),\psi^0(D))\in\Typ^0_r$ to be
\begin{equation*}
    f^0(D)(i)\defeq n(i)+\frac{2}{\cfactor}\cdot m(i-1)+l(i-2)\quad\text{and}\quad\psi^0(D)(i)\defeq\psi_1(i)\cdot\epsilon^{m(i-1)}\cdot\psi_2(i-2).
\end{equation*}
Now we want to prove that $\mathrm{str}(\Delta_r(x)(\delta(e,\chi)))$ is equal to
\begin{equation*}
    \sum_{D=(n,m,l,\psi_1,\psi_2)}\left(\prod_iQ_i(D)\right)\cdot(f^0(D),\psi^0(D))
\end{equation*}
where
\begin{equation*}
    Q_i(D)\defeq q^{(m(i)+\cfactor\cdot l(i))\sum_{j>i}(m(j)+\cfactor\cdot l(j))}x^{m(i)+\cfactor\cdot l(i)}Q\Typec{n(i)}{\psi_1(i)}{m(i)}{}{l(i)}{\psi_2(i)}
\end{equation*}
and where the sum is over the data $D$ as in \Cref{mainLemmaCount}, namely $n,m,l\colon\Z\to\Z_{\ge0},$ $\psi_1,\psi_2\colon\Z\to\mathrm{Sign}$ satisfying
\begin{equation*}
    e^0(i)=n(i)+\frac{2}{\cfactor}\cdot m(i)+l(i)\quad\text{and}\quad\chi^0(i)=\psi_1(i)\cdot\epsilon^{m(i)}\cdot\psi_2(i).
\end{equation*}

Since we have $\TypeC{a}{\psi_1}{a+1}{\psi_2}\equiv\TypeC{a+1}{\psi_2}{a}{\psi_1}\mod\Rel,$ (see \eqref{relation1inScase} in the case \Ogen), we have
\begin{equation*}
    (f^0(D),\psi^0(D))=\mathrm{str}\left(\cdots\star\delta^{(i+1)}(D)\star\delta^{(i)}(D)\star\cdots\right)
\end{equation*}
where
\begin{equation*}
    \delta^{(i)}(D)\defeq\Typec{(i+2)^{l(i)}}{\psi_2(i)}{(i+1)^{\frac{2}{\cfactor}\cdot m(i)}}{\epsilon^{m(i)}}{(i)^{n(i)}}{\psi_1(i)}.
\end{equation*}
With this, we can reduce the claim to the cases $e=(i,\cdots,i),$ which then immediately reduces to the case $e=(0,\cdots,0).$ This is possible because: i) for cases \Ugen, \Ogen, we have that if $\varepsilon\in\{0,1\}^r$ is $\varepsilon=(\cdots,\varepsilon^{(i+1)},\varepsilon^{(i)},\cdots),$ then $\inv(\varepsilon)=\sum_{i}\inv(\varepsilon)+\sum_{i>j}\lambda_1(\varepsilon^{(i)})\cdot\lambda_0(\varepsilon^{(j)}).$ Note that $\lambda_1(\varepsilon^{(i)})$ corresponds to $m(i)+l(i),$ while $\lambda_0(\varepsilon^{(i)})$ corresponds to $m(i)+l(i).$ ii) for the case \SPgen, we have that if $\varepsilon\in\{0,1,2\}^r$ is $\varepsilon=(\cdots,\varepsilon^{(i+1)},\varepsilon^{(i)},\cdots),$ then
\begin{equation*}
\begin{split}
    2\cdot\widetilde{\inv}(\varepsilon)+\binom{\lambda_1(\varepsilon)}{2}=&\sum_i\left(2\cdot\widetilde{\inv}(\varepsilon^{(i)})+\binom{\lambda_1(\varepsilon^{(i)})}{2}\right)+\sum_{i>j}\lambda_1(\varepsilon^{(i)})\lambda_1(\varepsilon^{(j)})\\&+\sum_{i>j}4\lambda_2(\varepsilon^{(i)})\lambda_0(\varepsilon^{(j)})+2\lambda_1(\varepsilon^{(i)})\lambda_0(\varepsilon^{(j)})+2\lambda_2(\varepsilon^{(i)})\lambda_1(\varepsilon^{(j)})\\
    =&\sum_i\left(2\cdot\widetilde{\inv}(\varepsilon^{(i)})+\binom{\lambda_1(\varepsilon^{(i)})}{2}\right)+\sum_{i>j}\Sigma(\varepsilon^{(i)})\cdot(2\lvert\varepsilon^{(j)}\rvert-\Sigma(\varepsilon^{(j)})).
\end{split}
\end{equation*}
Note that $\Sigma(\varepsilon^{(i)})$ corresponds to $m(i)+2\cdot l(i),$ while $2\lvert\varepsilon^{(i)}\rvert-\Sigma(\varepsilon^{(i)})$ corresponds to $m(i)+2\cdot n(i).$

Now we prove the claim for $e=(0,\ldots,0)$ by induction on $r.$ If we consider the expression for $\Delta_{r+1}(x)$ and separate the sum over $\varepsilon$ according to $\varepsilon_{r+1},$ we have that $\Delta_{r+1}(x)$ is
\begin{equation*}
\begin{array}{c|c}
    \Ugen&x\Delta_r(x)\star t(2)+\Delta_r(qx)\star t(0)\\
    \hline
    \Ogen&x\Delta_r(x)\star t(2)+\Delta_r(qx)\star t_{r+1}(0)\\
    \hline
    \SPgen&x^2\Delta_r(x)\star t(2)+(q+1)x\Delta_r(qx)\star t(1)+\Delta_r(q^2x)\star t(0)
\end{array}
\end{equation*}
Using the following straightening relations,
\begin{equation*}
\begin{array}{c|c}
    \Ugen&((0)^n,2)\equiv(1-(-q_0)^n)\cdot ((1)^2,(0)^{n-1})+(-q_0)^n\cdot(2,(0)^n)\\
    \hline
    \Ogen&\begin{array}{l}
    \displaystyle\TypeC{(0)^n}{\chi_1}{2}{\chi_2}\equiv(1-(q\epsilon)^{\lfloor n/2\rfloor}\chi_1(-\epsilon\chi_2)^n)\TypeC{(1)^2}{\epsilon}{(0)^{n-1}}{\epsilon\chi_1\chi_2}\\
    \displaystyle\quad\quad+(q\epsilon)^{\lfloor n/2\rfloor}\chi_1(-\epsilon\chi_2)^n
    \left(\frac{1+q\epsilon}{2}\TypeC{2}{\chi_2}{(0)^n}{\chi_1}+\frac{1-q\epsilon}{2}
    \TypeC{2}{(-1)^n\chi_2}{(0)^n}{(-1)^n\chi_1}\right)
    \end{array}\\
    \hline
    \SPgen&((0)^n,2)\equiv(1-q^{2n})\cdot((1)^2,(0)^{n-1})+q^{2n}\cdot(2,(0)^n)
\end{array}
\end{equation*}
the induction hypothesis implies that $\Delta_{r+1}(x)\left(\TypeCC{(0)^r}{\chi}\star\TypeCC{0}{\eta}\right)$ is congruent, modulo $\Rel_{r+1},$ to
\begin{equation*}
\begin{array}{c|c}
    \Ugen&\displaystyle\sum_{n+2m+l=r}Q(n,m,l)x^{m+l}\left(\begin{array}{@{}r@{}c@{}l@{}}&\displaystyle x(1-(-q_0)^n)&\typec{(2)^l}{(1)^{2(m+1)}}{(0)^{n-1}}\\+&\displaystyle x(-q_0)^n&\typec{(2)^{l+1}}{(1)^{2m}}{(0)^n}\\+&\displaystyle q^{m+l}&\typec{(2)^l}{(1)^{2m}}{(0)^{n+1}}\end{array}\right)\\
    \hline
    \Ogen&\begin{array}{l}\displaystyle\sum_{\substack{n+2m+l=r\\\psi_1\epsilon^m\psi_2=\chi}}Q\Typec{n}{\psi_1}{m}{}{l}{\psi_2}x^{m+l}\\
    \qquad\cdot\displaystyle\left(\begin{array}{@{}r@{}c@{}l@{}}
    &\displaystyle x(1-(q\epsilon)^{\lfloor n/2\rfloor}\psi_1(-\epsilon\eta)^n)&\Typec{(2)^l}{\psi_2}{(1)^{2(m+1)}}{\epsilon^{m+1}}{(0)^{n-1}}{\psi_1\epsilon\eta}\\
    +&\displaystyle x(q\epsilon)^{\lfloor n/2\rfloor}\psi_1(-\epsilon\eta)^n\frac{1+q\epsilon}{2}&\Typec{(2)^{l+1}}{\psi_2\eta}{(1)^{2m}}{\epsilon^m}{(0)^n}{\psi_1}\\
    +&\displaystyle x(q\epsilon)^{\lfloor n/2\rfloor}\psi_1(-\epsilon\eta)^n\frac{1-q\epsilon}{2}&\Typec{(2)^{l+1}}{(-1)^n\psi_2\eta}{(1)^{2m}}{\epsilon^m}{(0)^n}{(-1)^n\psi_1}\\
    +&\displaystyle q^{m+l}&\Typec{(2)^l}{\psi_2}{(1)^{2m}}{\epsilon^m}{(0)^{n+1}}{\psi_1\eta}\end{array}\right)\end{array}\\
    \hline
    \SPgen&\displaystyle\sum_{n+m+l=r}Q(n,m,l)x^{m+2l}\left(\begin{array}{@{}r@{}c@{}l@{}}&\displaystyle x^2(1-q^{2n})&\typec{(2)^l}{(1)^{m+2}}{(0)^{n-1}}\\+&\displaystyle x^2q^{2n}&\typec{(2)^{l+1}}{(1)^{m}}{(0)^n}\\+&\displaystyle xq^{m+2l}(q+1)&\typec{(2)^l}{(1)^{m+1}}{(0)^n}\\+&q^{2m+4l}&\typec{(2)^l}{(1)^{m}}{(0)^{n+1}}\end{array}\right)
\end{array}
\end{equation*}
Collecting terms, it remains to check that $Q\Typec{n}{\psi_1}{m}{}{l}{\psi_2}$ agrees with
\begin{equation*}
\begin{array}{c|c}
    \Ugen&\begin{array}{@{}r@{}c@{\,}l@{}}&(1-(-q_0)^{n+1})& Q(n+1,m-1,l)\\+&(-q_0)^n& Q(n,m,l-1)\\+&q_0^{2(m+l)}&Q(n-1,m,l)\end{array}\\
    \hline
    \Ogen&\begin{array}{@{}r@{}c@{\,}l@{}}&(1-(q\epsilon)^{\lfloor (n+1)/2\rfloor}\psi_1\epsilon\eta(-\epsilon\eta)^{n+1})& Q\Typec{n+1}{\psi_1\epsilon\eta}{m-1}{}{l}{\psi_2}\\
    +&(q\epsilon)^{\lfloor n/2\rfloor}\psi_1(-\epsilon\eta)^n\frac{1+q\epsilon}{2}& Q\Typec{n}{\psi_1}{m}{}{l-1}{\psi_2\eta}\\
    +&(q\epsilon)^{\lfloor n/2\rfloor}\psi_1(\epsilon\eta)^n\frac{1-q\epsilon}{2}& Q\Typec{n}{(-1)^n\psi_1}{m}{}{l-1}{(-1)^n\psi_2\eta}\\
    +&q^{m+l}& Q\Typec{n-1}{\psi_1\eta}{m}{}{l}{\psi_2}\end{array}\\
    \hline
    \SPgen&\begin{array}{@{}r@{}c@{\,}l@{}}&(1-q^{2(n+1)})& Q(n+1,m-2,l)\\+&q^{2n}& Q(n,m,l-1)\\
    +&q^{m+2l-1}(q+1)& Q(n,m-1,l)\\+&q^{2m+4l}& Q(n-1,m,l)\end{array}
\end{array}
\end{equation*}
Dividing by $Q\Typec{n}{\psi_1}{m}{}{l}{\psi_2},$ and using that $\frac{\lfloor k/2\rfloor_{q^2}}{e(k,\chi_1)}\left\slash\frac{\lfloor(k-1)/2\rfloor_{q^2}}{e(k-1,\chi_1\chi_2)}\right.=q^{\lfloor k/2\rfloor}-\epsilon^{\lfloor k/2\rfloor}\chi_1(-\chi_2)^k$ for the case \Ogen, we are looking at
\begin{equation*}
\begin{array}{c|c}
    \Ugen&\displaystyle\frac{1}{(-q_0)^{n+2m+l}-1}\cdot\left(\begin{array}{@{}r@{}l@{}}&(q_0^{2m}-1)(-q_0)^l\\+&((-q_0)^l-1)\\+&(-q_0)^{2m+l}((-q_0)^n-1)\end{array}\right)\\
    \hline
    \Ogen&\begin{array}{l}\displaystyle\frac{1}{q^{\lfloor (n+2m+l)/2\rfloor}-\epsilon^{\lfloor(n+l)/2\rfloor}\psi_1\psi_2(-\eta)^{n+l}}\\
    \qquad\cdot\displaystyle\left(\begin{array}{@{}r@{}l@{}}&q^{\lfloor(n+1)l/2\rfloor-\lfloor nl/2\rfloor}\epsilon^{\lfloor(n+1)/2\rfloor}\psi_1(-\epsilon\eta)^n(q^m-1)\\
    +&q^{\lfloor n(l-1)/2\rfloor-\lfloor nl/2\rfloor}(q\epsilon)^{\lfloor n/2\rfloor}\psi_1(-\epsilon\eta)^n\frac{1+q\epsilon}{2}(q^{\lfloor l/2\rfloor}-\epsilon^{\lfloor l/2\rfloor}\psi_2(-\eta)^l)\\
    +&q^{\lfloor n(l-1)/2\rfloor-\lfloor nl/2\rfloor}(q\epsilon)^{\lfloor n/2\rfloor}(-1)^n\psi_1(-\epsilon\eta)^n\frac{1-q\epsilon}{2}(q^{\lfloor l/2\rfloor}-\epsilon^{\lfloor l/2\rfloor}\psi_2(-\eta(-1)^n)^l)\\
    +&q^{\lfloor (n-1)l/2\rfloor-\lfloor nl/2\rfloor}q^{m+l}(q^{\lfloor n/2\rfloor}-\epsilon^{\lfloor n/2\rfloor}\psi_1(-\eta)^n)\end{array}\right)\end{array}\\
    \hline
    \SPgen&\displaystyle\frac{1}{q^{2(n+m+l)}-1}\cdot\left(\begin{array}{@{}r@{}l@{}}-&q^{2l}(q^{m}-1)(q^{m-1}-1)\\+&(q^{2l}-1)\\+&q^{m+2l-1}(q+1)(q^m-1)\\+&q^{2m+2l}(q^{2n}-1)\end{array}\right)
\end{array}
\end{equation*}
which simplify to $1$ after a routine computation.
\end{proof}

%% file: 5-HeckeStructure.tex
\section{Analysis of the Hecke module structure}\label{HeckeStructureSection}
We keep the notation from the previous section. In particular, $R$ is a coefficient ring with $\#\mathrm{Sign}\in R^\times.$
\subsection{Explicit bases}
Let $r,\alpha\ge 1$ be positive integers.
\begin{definition}
We consider the following total preorder. For $e,f\in\Z^r,$ we say $e\prec f$\index{0<@$\prec$} if there exists $1\le i\le r$ such that i) $e_j-\Sigma(e)/r=f_j-\Sigma(f)/r$ for $j<i,$ and ii) $e_i-\Sigma(e)/r<f_i-\Sigma(f)/r.$ We say $e\preccurlyeq f$ if either $e\prec f$ or $e_i-\Sigma(e)/r=f_i-\Sigma(f)/r$ for all $i.$
\end{definition}
Note that $e\preccurlyeq f$ and $f\preccurlyeq e$ if and only if there is $k\in\Z$ with $e=(f_1+k,\ldots,f_r+k).$

We denote
\index{Typ0<e@$\Typ^{0,\prec e}_r,\ \Typ^{0,\preccurlyeq e}_r$}
\begin{equation*}
    \Typ^{0,\preccurlyeq e}_r\defeq\{\delta(f,\chi)\in\Typ^0_r\colon f\preccurlyeq e\},\quad \Typ^{0,\prec e}_r\defeq\{\delta(f,\chi)\in\Typ^0_r\colon f\prec e\}.
\end{equation*}

Suppose that we have commuting operators $S_0,\ldots,S_r\colon\mathcal{S}(X/K,R)\to \mathcal{S}(X/K,R)$ such that their duals are of the form
\begin{equation*}
    \sum_{k=0}^rS_k^*x^k=\sum_{\varepsilon\in\{0,1\}^r}a(\varepsilon)x^{\lambda_1(\varepsilon)}t(\alpha\cdot\varepsilon)
\end{equation*}
for certain coefficients $a(\varepsilon)\in R.$ We assume that $a(0,\ldots,0,\underbrace{1,\ldots,1}_{k})=1$ for all $0\le k\le r.$

\begin{proposition}\label{largestInSk}
Let $f=(f_1,\ldots,f_r)$ with $f_1\ge\cdots\ge f_r.$ Denote
\begin{equation*}
    s_k(f)\defeq f-\alpha\cdot(0,\ldots,0,\underbrace{1,\ldots,1}_{k}).
\end{equation*}
Assume that $\delta_s(f)\not\in\Rel.$ Then we have that
\begin{equation*}
    S_k(\delta_s(f))\equiv\delta_s(s_k(f))\mod R[\Typ^{0,\prec s_k(f)}_r].
\end{equation*}
\end{proposition}
\begin{proof}
Consider $\delta_s(e)\in\mathrm{Supp}(S_k(\delta_{s'}(f))).$ This is the same as having $\delta_{s'}(f)\in\mathrm{Supp}(S_k^*(\delta_s(e))).$ In the expression for $S_k^*$ above, the terms $\delta_s(e+\alpha\cdot\varepsilon)$ get straightened to terms that still lie inside the cube $[e_1,e_1+\alpha]\times\cdots\times[e_r,e_r+\alpha].$ We also have $\Sigma(f)=\Sigma(e)+\alpha\cdot k.$ Switching this around, we get that
\begin{equation*}
    e\in[f_1-\alpha,f_1]\times\cdots[f_r-\alpha,f_r],\quad\Sigma(e)=\Sigma(f)-\alpha\cdot k.
\end{equation*}
Now it is easy to see that the point $e$ satisfying the above which is maximal with respect to $\preccurlyeq$ is simply $e=s_k(f).$ The claim now follows from the assumption that $a(0,\ldots,0,\underbrace{1,\ldots,1}_{k})=1.$
\end{proof}

\begin{theorem}\label{basisTheorem}
Consider the set
\begin{equation*}
    \mathcal{B}\defeq \left\{\delta_s(e)\colon s\colon\mathrm{Sign}^r\to\{\pm1\},\ (e_{i}-e_{i+1}-\left(1\text{ if }s_i\neq s_{i+1}\right))\in\left[0,\alpha-1\right],\ 0\le e_r<\alpha\right\}.
\end{equation*}
Under this basis, $\mathcal{S}(X/K,R)$ is a free $R[S_1,\ldots,S_r,S_r^{-1}]$-module.
\end{theorem}
\begin{proof}
By \Cref{largestInSk}, for any $a_1,\ldots,a_{r-1}\ge0$ and $a_r\in\Z,$ we have
\begin{equation*}
    S_1^{a_1}\cdots S_{r-1}^{a_{r-1}}S_r^{a_r}(\delta_s(f))\equiv \delta_s(e)\mod\Z[\Typ^{0,\prec e}_r]
\end{equation*}
where $e_i=f_i-\alpha\cdot(a_r+\cdots+a_{r-i+1}).$ Conversely, given $e\in\Typ^0$ and $\delta_s(e)$ nonzero, there is a unique $\delta_s(f)\in\mathcal{B}$ and integers $a_1,\ldots,a_r$ as above with $S_1^{a_1}\cdots S_{r-1}^{a_{r-1}}S_r^{a_r}(\delta_s(f))\equiv\delta_s(e)\mod\Z[\Typ^{\prec e}].$ Together with \Cref{basisremark}, this implies that $\mathcal{B}$ is a basis of $\mathcal{S}(X/K,R)$ as a $\mathcal{H}(G,K)$-module.
\end{proof}

\subsection{Case \texorpdfstring{\Ugen}{(uH)}}
Assume we are in the case \Ugen for this subsection. We will give another treatment of \cite[Theorem 2]{HironakaHerm}.
\begin{definition}
Consider $\Delta^{1/2}_r(x)\colon R[\Typ_r]\to R[x][\Typ_r]$ given by
\begin{equation*}
    \Delta^{1/2}_r(x)\defeq\sum_{\varepsilon\in\{0,1\}^r}(-q_0)^{\inv(\varepsilon)}x^{\lambda_1(\varepsilon)}t(\varepsilon)
\end{equation*}
We also let $\Delta^{1/2}_{k,r}\colon R[\Typ_r]\to R[\Typ_r]$ be $\Delta^{1/2}_{k,r}\defeq [x^k](\Delta^{1/2}_r(x)).$
\end{definition}
\begin{proposition}
$\Delta^{1/2}_{k,r}$ preserves $\mathrm{Gr}_r(\Rel).$
\end{proposition}
\begin{proof}
This follows at once from \Cref{preserveshalf}.
\end{proof}
\begin{corollary}
    Let $S_{k,r}$ denote the adjoints of $\Delta_{k,r}^{1/2}.$ Then $\mathcal{S}(X/K,R)$ is free of rank one as a $R[S_{1,r},\ldots,S_{r,r},S_{r,r}^{-1}]$-module, generated by $\delta(0,\ldots,0).$
\end{corollary}
\begin{proof}
    This follows at once from \Cref{basisTheorem} and the above.
\end{proof}
Note that this equips $\mathcal{S}(X/K,R)$ with an $R$-algebra structure.
\begin{theorem}[Hironaka's spherical transform]\label{RSat-uH}
    Suppose $R$ contains a square root of $-q_0,$ denoted by $\sqrt{-q_0}.$ Let $\mu_1,\ldots,\mu_r,\nu_1,\ldots,\nu_r$ be indeterminates. Then we have a commutative diagram of $R$-algebras
    \begin{equation*}
        \begin{tikzcd}
            \mathcal{H}(G,K)\arrow[r,"\mathrm{Sat}"]\arrow[d,hook]&R[\mu_1^{\pm1},\ldots,\mu_r^{\pm1}]^{\mathrm{sym}}\arrow[d,hook]\\
            \mathrm{S}(X/K,R)\arrow[r,"\mathrm{Sat}^{1/2}"]&R[\nu_1^{\pm1},\ldots,\nu_r^{\pm1}]^{\mathrm{sym}}
        \end{tikzcd}
    \end{equation*}
    where $\mathrm{Sat}$ denotes the Satake transform $\mathrm{Sat}(T_{k,r})=(-q_0)^{k(r-k)}\sigma_k(\mu_1,\ldots,\mu_r),$ $\mathrm{Sat}^{1/2}$ is defined by $\mathrm{Sat}^{1/2}(S_{k,r})\defeq \sqrt{-q_0}^{k(r-k)}\sigma_k(\nu_1,\ldots,\nu_r),$ and the rightmost map is induced by $\mu_i\mapsto\nu_i^2$ for $1\le i\le r.$
\end{theorem}
\begin{proof}
Consider
\begin{equation*}
    A_r(x)\defeq\sum_{k=0}^r(-q_0)^{-k(r-k)}\Delta_{k,r}\cdot (-x)^k,\quad B_r(x)\defeq\sum_{k=0}^r\sqrt{-q_0}^{-k(r-k)}\Delta^{1/2}_{k,r}\cdot (-x)^k.
\end{equation*}
This is so that $\mathrm{Sat}(A_r(x))=\prod_{i=1}^r(1-x\mu_i)$ and $\mathrm{Sat}^{1/2}(B_r(x))=\prod_{i=1}^r(1-x\nu_i).$ So it remains to see that $A_r(x^2)\isequal B_r(x)B_r(-x).$

We have
\begin{equation*}
    A_{r+1}(x)=A_r(-xq_0)\star\left((-q_0)^{-r}(-x)t_{r+1}(2)+t_{r+1}(0)\right)
\end{equation*}
as well as
\begin{equation*}
    B_{r+1}(x)=B_r(x\sqrt{-q_0})\star\left(\sqrt{-q_0}^{-r}(-x)t_{r+1}(1)+t_{r+1}(0)\right).
\end{equation*}
Now the claim follows by induction on $r$ as
\begin{equation*}
    \left(\sqrt{-q_0}^{-r}x\cdot t_{r+1}(1)+t_{r+1}(0)\right)\left(-\sqrt{-q_0}^{-r}x\cdot t_{r+1}(1)+t_{r+1}(0)\right)=-(-q_0)^{-r}x\cdot t_{r+1}(2)+t_{r+1}(0).\qedhere
\end{equation*}
\end{proof}
\begin{remark}\label{uHsphremark}
It is not a priori clear how the above map $\mathrm{Sat}^{1/2}$ in the case $R=\C$ is related to Hironaka's spherical transform of \cite[Theorem 2]{HironakaHerm}. In the case $r=2,$ we can explicitly see by the proof of \cite[\S1, Theorem 4]{Hironaka2} that, up to normalization, Hironaka's transform $F\colon\mathcal{S}(X/K,\C)\rightiso\C[q_0^{\pm z_1},q_0^{\pm z_2}]$ agrees with $\mathrm{Sat}^{1/2}$ under the identification $\nu_i\mapsto \sqrt{-1}\cdot q_0^{z_i}.$\footnote{We note that our $q_0$ corresponds to ``$q$'' in Hironaka's notation. Moreover, our usage of $z_i$ above is following the notation in \cite{HironakaHerm}. We note, however, that in \cite{Hironaka2} the ``$z_i$'' denote half of the previous ones.} In general, it should be possible to prove that they agree by reproving the inverse Satake transform \Cref{InvSat} from Hironaka's work.
\end{remark}

\subsection{Case \texorpdfstring{\Ogen}{(S)}}
Assume we are in the case \Ogen for this subsection. The following confirms a conjecture of Hironaka from \cite{Hironaka2}.
\begin{corollary}\label{explicitBasis}
$\mathcal{S}(X/K,R)$ is a free $\mathcal{H}(G,K)$-module of rank $4^r,$ with basis
\begin{equation*}
    \mathcal{B}\defeq \left\{\delta_s(e)\colon s\colon\mathrm{Sign}^r\to\{\pm1\},\ (e_{i}-e_{i+1}-\left(1\text{ if }s_i\neq s_{i+1}\right))\in\{0,1\},\ e_r\in\{0,1\}\right\}.
\end{equation*}
\end{corollary}
\begin{proof}
That $\mathcal{B}$ is a basis follows by \Cref{explicitBasis} by taking $S_k=T_{k,r}$ and $\alpha=2.$ For each $s\in\mathrm{Sign}^r\to\{\pm1\},$ there are $2^r$ possibilities for $e$ such that $\delta_s(e)\in\mathcal{B}.$ Thus $\#\mathcal{B}=2^r\cdot\#(\mathrm{Sign}^r)=4^r.$
\end{proof}
\begin{remark}
For the case $r=2,$ this is essentially the same basis obtained by Hironaka in \cite[\S 4, Theorem 4]{Hironaka2}.
\end{remark}

\subsection{Case \texorpdfstring{\SPgen}{(A)}}
Assume we are in the case \SPgen for this subsection. We will give another treatment of \cite[Theorem 1]{HironakaAlt}.

\begin{definition}
Consider $\Delta_r^{1/2}(x)\colon R[\Typ_r]\to R[x][\Typ_r]$ given by
\begin{equation*}
    \Delta_r^{1/2}(x)\defeq\sum_{\varepsilon\in\{0,1\}^{r}}q^{2\inv(\varepsilon)}x^{\lambda_1(\varepsilon)}t(\varepsilon).
\end{equation*}
We also let $\Delta_{k,r}^{1/2}\colon R[\Typ_r]\to R[\Typ_r]$ to be $\Delta_{k,r}^{1/2}\defeq [x^k](\Delta_r^{1/2}(x)).$
\end{definition}
\begin{proposition}
    $\Delta_{k,r}^{1/2}$ preserves $\mathrm{Gr}_r(\Rel).$
\end{proposition}
\begin{proof}
This follows at once from \Cref{preserveshalf}.
\end{proof}
\begin{corollary}
    Let $S_{k,r}$ denote the adjoints of $\Delta_{k,r}^{1/2}.$ Then $\mathcal{S}(X/K,R)$ is free of rank one as a $R[S_{1,r},\ldots,S_{r,r},S_{r,r}^{-1}]$-module, generated by $\delta(0,\ldots,0).$
\end{corollary}
\begin{proof}
    This follows at once from \Cref{basisTheorem} and the above.
\end{proof}
Note that this equips $\mathcal{S}(X/K,R)$ with an $R$-algebra structure.

\begin{theorem}[Hironaka--Sato's spherical transform]
    Suppose $R$ contains a square root of $q,$ denoted by $\sqrt{q}.$ Let $\mu_1,\ldots,\mu_{2r},\nu_1,\ldots,\nu_r$ be indeterminates. Then we have a commutative diagram of $R$-algebras
    \begin{equation*}
        \begin{tikzcd}
            \mathcal{H}(G,K)\arrow[r,"\mathrm{Sat}"]\arrow[d,twoheadrightarrow]&R[\mu_1^{\pm1},\ldots,\mu_{2r}^{\pm1}]^{\mathrm{sym}}\arrow[d,twoheadrightarrow]\\
            \mathrm{S}(X/K,R)\arrow[r,"\mathrm{Sat}^{1/2}"]&R[\nu_1^{\pm1},\ldots,\nu_r^{\pm1}]^{\mathrm{sym}}
        \end{tikzcd}
    \end{equation*}
    where $\mathrm{Sat}$ denotes the Satake transform $\mathrm{Sat}(T_{k,r})=\sqrt{q}^{k(2r-k)}\sigma_k(\mu_1,\ldots,\mu_{2r}),$ $\mathrm{Sat}^{1/2}$ is defined by $\mathrm{Sat}^{1/2}(S_{k,r})\defeq q^{k(r-k)}\sigma_k(\nu_1,\ldots,\nu_r),$ and the rightmost map is induced by
    \begin{equation*}
        \mu_{2k-1}+\mu_{2k}\mapsto\left(\sqrt{q}+\frac{1}{\sqrt{q}}\right)\nu_k\quad\text{and}\quad\mu_{2k-1}\mu_{2k}\mapsto \nu_k^2\quad\quad\text{for }1\le k\le r.
    \end{equation*}
\end{theorem}
\begin{proof}
Consider
    \begin{equation*}
    A_r(x)\defeq\sum_{k=0}^{2r}\sqrt{q}^{-k(2r-k)}\Delta_{k,r}(-x)^k,\quad B_r(x)\defeq\sum_{k=0}^rq^{-k(r-k)}\Delta^{1/2}_{k,r}(-x)^k.
\end{equation*}
This is so that $\mathrm{Sat}(A_r(x))=\prod_{i=1}^{2r}(1-x\mu_i)$ and $\mathrm{Sat}^{1/2}(B_r(x))=\prod_{i=1}^r(1-x\nu_i).$ So it remains to see that $A_r(x)\isequal B_r(x\sqrt{q})B_r(x/\sqrt{q}).$

We have
\begin{equation*}
    A_{r+1}(x)=A_r(xq)\star\left(x^2q^{-2r}t_{r+1}(2)-xq^{-r}(q^{1/2}+q^{-1/2})t_{r+1}(1)+t_{r+1}(0)\right)
\end{equation*}
as well as
\begin{equation*}
    B_{r+1}(x)=B_r(xq)\star(-xq^{-r}t_{r+1}(1)+t_{r+1}(0)).
\end{equation*}
Now the claim follows by induction on $r$ as
\begin{equation*}
\begin{split}
    &(-xq^{-r+1/2}t_{r+1}(1)+t_{r+1}(0))\cdot(-xq^{-r-1/2}t_{r+1}(1)+t_{r+1}(0))\\
    &\quad\quad=\left(x^2q^{-2r}t_{r+1}(2)-xq^{-r}(q^{1/2}+q^{-1/2})t_{r+1}(1)+t_{r+1}(0)\right).\qedhere
\end{split}
\end{equation*}
\end{proof}
\begin{remark}
It is easy to see that, in the case $R=\C,$ the above map $\mathrm{Sat}^{1/2}$ agrees with Hironaka--Sato's spherical transform of \cite[Theorem 1]{HironakaAlt}. Namely, by \cite[(1.14) and Lemma 2.2]{HironakaAlt}, their transform $\mathcal{S}(X/K,\C)\rightiso\C[q^{\pm z_1},\ldots,q^{\pm z_n}]^{\mathrm{sym}}$ is exactly $\mathrm{Sat}^{1/2}$ under the identification $\nu_i\mapsto q^{z_i}.$
\end{remark}

%% file: 6-Appendix.tex
\section{Relation with relative Satake isomorphism}\label{Section-Appendix}
In this appendix we explain how the explicit description \Cref{inducedThm} follows from the inverse Satake transform \cite[Theorem 7.7]{Sakellaridis-InverseSatake} in the cases it applies.

Let $G$ be an unramified reductive group over a local non-archimedean field $F_0$ with residue cardinality $q_0$ and $X$ be a spherical $G$-variety. Let $K\subseteq G(F_0)$ be a hyperspecial subgroup and $A\subseteq G$ a Cartan subtorus. We denote $\Lambda\defeq X_*(A)$ its cocharacter group and $W$ its Weyl group. Under the assumptions of \cite{Sakellaridis-InverseSatake}, we have the following data
\begin{enumerate}
    \item A Cartan subtorus $A_X\subseteq X,$ for which we denote its cocharacter group by $\Lambda_X\defeq X_*(A_X).$ We denote $\Lambda_X^+\subseteq\Lambda_X$ the subset of antidominant elements. This is such that we have a relative Cartan decomposition
    \begin{equation*}
    \begin{split}
        X(F_0)/K&\longleftrightarrow\Lambda_X^+,\\
        x_{\check{\lambda}}K&\longleftrightarrow\check{\lambda}
        \end{split}
    \end{equation*}
    We use this to identify $\mathcal{S}(X/K,\C)=\C[\Lambda_X^+].$
    \item A root system $\check{\Phi}_X$ on $\Lambda_X$ with Weyl group $W_X$ and simple roots $\Delta_X.$
    \item A $W_X$-invariant multiset of triples $\Theta=\{(\check{\theta},\sigma_{\check{\theta}},r_{\check{\theta}})\}$ where $\check{\theta}\in \Lambda_X,$ $\sigma_{\check{\theta}}\in\{\pm1\}$ and $r_{\check{\theta}}\in\frac{1}{2}\Z.$ There is a certain sub-multiset $\Theta^+\subseteq\Theta$ of positive elements such that $\Theta=\Theta^+\sqcup\Theta^-$ for $\Theta^-=\{(-\check{\theta},\sigma_{\check{\theta}},r_{\check{\theta}})\colon (\check{\theta},\sigma_{\check{\theta}},r_{\check{\theta}})\in\Theta^+\}.$
\end{enumerate}
We consider two different normalizations for a basis of $\mathcal{S}(X/K,\C)=\C[\Lambda_X^+]$: for $\check{\lambda}\in\Lambda_X,$ we let $\delta_{\check{\lambda}}\in\C[\Lambda_X]$ denote the element corresponding to $\mathrm{char}(x_{\check{\lambda}}K),$ and we denote $e^{\check{\lambda}}\defeq q_0^{\langle\check{\lambda},\rho_{P(X)}\rangle}\cdot\delta_{\check{\lambda}}.$ This let us write the group structure on $\Lambda_X$ multiplicatively: $e^{\check{\lambda}_1}e^{\check{\lambda}_2}=e^{\check{\lambda}_1+\check{\lambda}_2}.$ We consider the action $W_X$ on $\C[\Lambda_X]$ also in terms of the latter basis: $w(e^{\check{\lambda}})\defeq e^{w(\check{\lambda})}.$
\begin{theorem}[Relative Satake isomorphism]
    Denote $\mathcal{H}_X\defeq\C[\Lambda_X]^{W_X}.$ There is a relative Satake isomorphism
    \begin{equation*}
        \mathrm{Sat}\colon \mathcal{S}(X/K,\C)\rightiso\mathcal{H}_X
    \end{equation*}
    which is equivariant for the $\mathcal{H}(G,K,\C)$-action, where the action on the right hand side is given by
    \begin{equation*}
        \mathcal{H}(G,K,\C)\iso\C[\check{A}]^W\to\C[\delta_{(X)}^{\frac{1}{2}}\check{A}_X]^{W_X}\simeq\mathcal{H}_X.
    \end{equation*}
\end{theorem}
\begin{theorem}[Inverse Satake transform]
    If $h\in\mathcal{H}_X=\C[\Lambda_X]^{W_X},$ then $\mathrm{Sat}^{-1}(h)\in\mathcal{S}(X/K,\C)=\C[\Lambda_X^+]$ is the function
    \begin{equation*}
        \mathrm{Sat}^{-1}(h)=\left(h\cdot L_X^{\frac{1}{2}}\right)\rvert_{\Lambda_X^+}
    \end{equation*}
    where
    \begin{equation*}
        L_X^{\frac{1}{2}}\defeq\frac{\prod_{\check{\gamma}\in\check{\Phi}_X^+}(1-e^{\check{\gamma}})}{\prod_{\check{\theta}\in\Theta^+}(1-\sigma_{\check{\theta}}q^{-r_{\check{\theta}}}e^{\check{\theta}})}
    \end{equation*}
    is expanded as a power series supported on the translates of a certain cone $\mathcal{C}_X.$
\end{theorem}
\begin{definition}
    We consider the pairings $(\cdot,\cdot)_e,(\cdot,\cdot)_\delta\colon\C[\Lambda_X]\times\C[\Lambda_X]\to\C$ given by
    \begin{equation*}
        (x,h)_?\defeq\langle x,h\cdot L_X^{\frac{1}{2}}\rangle_?,\quad ?\in\{e,\delta\}
    \end{equation*}
    where $\langle\cdot,\cdot\rangle_?\colon\C[\Lambda_X]\times\C[\Lambda_X]\to\C$ denote the pairings given by
    \begin{equation*}
        \langle e^{\check{\lambda}_1},e^{\check{\lambda}_2}\rangle_e=\begin{cases}1&\text{if }\check{\lambda}_1=\check{\lambda}_2\\0&\text{otherwise,}\end{cases},\quad\langle \delta_{\check{\lambda}_1},\delta_{\check{\lambda}_2}\rangle_\delta=\begin{cases}1&\text{if }\check{\lambda}_1=\check{\lambda}_2\\0&\text{otherwise.}\end{cases}
    \end{equation*}
    We denote $\mathrm{Rel}_?\subseteq\C[\Lambda_X]$ to be the left annihilator of $\C[\Lambda_X]^{W_X}$ under $(\cdot,\cdot)_?.$
\end{definition}
\begin{remark}\label{etodelta}
    If we consider the renormalization map $f\colon\C[\Lambda_X]\to\C[\Lambda_X]$ given by
    \begin{equation*}
        f(e^{\check{\lambda}})=q_0^{-\langle\check{\lambda},\rho_{P(X)}\rangle}\cdot\delta_{\check{\lambda}},
    \end{equation*}
    then we have $\langle x,h\rangle_e=\langle f(x),h\rangle_\delta$ and $(x,h)_e=(f(x),h)_\delta.$ In particular, $\mathrm{Rel}_\delta=f(\mathrm{Rel}_e).$
\end{remark}
\begin{proposition}
    Let $?\in\{e,\delta\}.$ We have that $\C[\Lambda_X^+]\hookrightarrow\C[\Lambda_X]$ induces an injection $\mathrm{str}_?^{-1}\colon \C[\Lambda_X^+]\hookrightarrow \C[\Lambda_X]/\mathrm{Rel}_?.$ Moreover, $\mathrm{Rel}_?$ is preserved by the action of $\mathcal{H}_X,$ and if $h\in\mathcal{H}_X,$ we have the commutative diagrams
    \begin{equation*}
        \begin{tikzcd}
            \mathcal{S}(X/K,\C)\arrow[d,"h^*_?"]\arrow[r,"\mathrm{str}_?^{-1}"]&\C[\Lambda_X]/\mathrm{Rel}_?\arrow[d,"\cdot\hat{h}^*_?"]\\
            \mathcal{S}(X/K,\C)\arrow[r,"\mathrm{str}_?^{-1}"]&\C[\Lambda_X]/\mathrm{Rel}_?
        \end{tikzcd}
    \end{equation*}
    where: i) $h^*_?$ is the adjoint of $h$ under $\langle\cdot,\cdot\rangle_?,$ ii) $\hat{h}^*_?\in\C[\Lambda_X]$ is the adjoint of $\hat{h}$ under $\langle\cdot,\cdot\rangle_?,$ namely:
    \begin{equation*}
        \hat{h}=\sum_{\check{\lambda}}a_{\check{\lambda}}e^{\check{\lambda}}\implies\hat{h}^*_e=\sum_{\check{\lambda}}a_{\check{\lambda}}e^{-\check{\lambda}}
    \end{equation*}
    and
    \begin{equation*}
        \hat{h}=\sum_{\check{\lambda}}b_{\check{\lambda}}\delta_{\check{\lambda}}\implies\hat{h}^*_\delta=\sum_{\check{\lambda}}b_{\check{\lambda}}\delta_{-\check{\lambda}}
    \end{equation*}
\end{proposition}
\begin{proof}
   The inverse Satake transform says that the composition
    \begin{equation*}
        \mathcal{S}(X/K,\C)\times\mathcal{S}(X/K,\C)\xrightiso{\mathrm{id}\times\mathrm{Sat}}\mathcal{S}(X/K,\C)\times\mathcal{H}_K\iso\C[\Lambda_X^+]\times\C[\Lambda_X]^{W_X}\xrightarrow{(\cdot,\cdot)_?}\C
    \end{equation*}
    is equal to the pairing $\langle\cdot,\cdot\rangle_?.$ This is nondegenerate, and thus induces $\mathrm{str}_?^{-1}\colon \C[\Lambda_X^+]\hookrightarrow \C[\Lambda_X]/\mathrm{Rel}_?.$ The commutative diagrams are now clear from this nondegeneracy.
\end{proof}
Now we explain how the module $\mathrm{Rel}_e$ can be determined in terms of the combinatorial data $\Theta.$
\begin{proposition}
    For $\check{\lambda}\in\Lambda_X$ and $\alpha\in\Delta_X,$ we consider the elements
    \begin{equation*}
        \mathrm{Rel}_e(\check{\lambda},\alpha)\defeq (B_\alpha)^*_e\cdot\left(e^{\check{\lambda}}+e^{\check{\alpha}+w_\alpha(\check{\lambda})}\right)
    \end{equation*}
    where
    \begin{equation*}
        B_\alpha\defeq\prod_{\substack{\check{\theta}\in\Theta^+\\-w_\alpha\check{\theta}\in\Theta^+}}(1-\sigma_{\check{\theta}}q_0^{-r_{\check{\theta}}}e^{\check{\theta}}).
    \end{equation*}
    Then for all $\check{\lambda}\in\Lambda_X$ and $\alpha\in\Delta_X,$ we have $\mathrm{Rel}_e(\check{\lambda},\alpha)\in\mathrm{Rel}_e.$
\end{proposition}
\begin{proof}
    The fact that $\mathrm{Rel}_e(\check{\lambda},\alpha)\in\mathrm{Rel}_e$ follows at once from noting that the element
    \begin{equation*}
        B_{\alpha}\cdot L_X^{\frac{1}{2}}=\frac{\prod_{\check{\gamma}\in\check{\Phi}_X^+}(1-e^{\check{\gamma}})}{\prod_{\substack{\check{\theta}\in\Theta^+\\w_\alpha\check{\theta}\in\Theta^+}}(1-\sigma_{\check{\theta}}q_0^{-r_{\check{\theta}}}e^{\check{\theta}})}
    \end{equation*}
    satisfies that
    \begin{equation*}
        w_\alpha\left(B_\alpha\cdot L_X^{\frac{1}{2}}\right)=\frac{w_\alpha\left(\prod_{\check{\gamma}\in\check{\Phi}_X^+}(1-e^{\check{\gamma}})\right)}{\prod_{\substack{\check{\theta}\in\Theta^+\\w_\alpha\check{\theta}\in\Theta^+}}(1-\sigma_{\check{\theta}}q_0^{-r_{\check{\theta}}}e^{\check{\theta}})}=\frac{(1-e^{-\check{\alpha}})\cdot\prod_{\substack{\check{\gamma}\in\check{\Phi}_X^+\\\check{\gamma}\neq\check{\alpha}}}(1-e^{\check{\gamma}})}{\prod_{\substack{\check{\theta}\in\Theta^+\\w_\alpha\check{\theta}\in\Theta^+}}(1-\sigma_{\check{\theta}}q_0^{-r_{\check{\theta}}}e^{\check{\theta}})}=-e^{-\check{\alpha}}\cdot B_\alpha\cdot L_X^{\frac{1}{2}},
    \end{equation*}
    as then for any $\hat{h}\in\C[\Lambda_X]^{W_X}$ we have that $((B_\alpha)^*_ee^{\check{\lambda}},\hat{h})_e$ is
    \begin{equation*}
        \langle e^{\check{\lambda}},\hat{h}\cdot B_\alpha\cdot L_X^{\frac{1}{2}}\rangle_e=\langle w_\alpha(e^{\check{\lambda}}),-\hat{h}\cdot e^{-\check{\alpha}}B_\alpha\cdot L_X^{\frac{1}{2}}\rangle_e=\langle -e^{w_\alpha(\check{\lambda})+\check{\alpha}},\hat{h}\cdot B_\alpha\cdot L_X^{\frac{1}{2}}\rangle_e=(-(B_\alpha)^*_ee^{w_\alpha(\check{\lambda})+\check{\alpha}},\hat{h})_e
    \end{equation*}
    and thus $(\mathrm{Rel}(\check{\lambda},\alpha),\hat{h})_e=0.$
\end{proof}

\begin{remark}
    In the situations where \cite{Sakellaridis} applies, we have
    \begin{equation*}
        B_\gamma=\begin{cases}
            1-q_0^{-\langle\check{\gamma},\rho_{P(X)}\rangle} e^{\check{\gamma}}&\text{type $G$,}\\
           (1-q_0^{\langle\check{\theta},\rho_{P(X)}\rangle-\langle\check{\rho},\gamma\rangle}e^{\check{\theta}})(1-q_0^{-\langle\check{\theta},\rho_{P(X)}\rangle}e^{-w_\gamma\check{\theta}}) &\text{type $T$ split,}\\
            (1+q_0^{\frac{1}{2}\langle\check{\gamma},\rho_{P(X)}\rangle-\langle\check{\rho},\gamma\rangle}e^{\frac{\check{\gamma}}{2}})(1-q_0^{-\frac{1}{2}\langle\check{\gamma},\rho_{P(X)}\rangle}e^{\frac{\check{\gamma}}{2}})&\text{type $T$ non-split,}\\
            1&\text{type $(U,\psi).$}
    \end{cases}
    \end{equation*}
    where in type $T$ split, $\check{\theta}\in\Lambda_X$ is a certain element such that $\check{\theta}-w_\gamma\check{\theta}=\check{\gamma}.$ For example, if $\gamma$ is a root of type $G,$ we have
    \begin{equation*}
        \mathrm{Rel}_e(\check{\lambda},\gamma)=e^{\check{\lambda}}-q_0^{-\langle\check{\gamma},\rho_{P(X)}\rangle}e^{\check{\lambda}-\check{\gamma}}+e^{\check{\gamma}+w_\gamma(\check{\lambda})}-q_0^{-\langle\check{\gamma},\rho_{P(X)}\rangle}e^{w_\gamma(\check{\lambda})},
    \end{equation*}
    which by \Cref{etodelta} correspond to the elements $\mathrm{Rel}_\delta(\check{\lambda},\gamma)\in\mathrm{Rel}_\delta$ given by
    \begin{equation*}
        \mathrm{Rel}_\delta(\check{\lambda},\gamma)=q_0^{-\langle\check{\lambda},\rho_{P(X)}\rangle}\left(\delta_{\check{\lambda}}-\delta_{\check{\lambda}-\check{\gamma}}-q_0^{\langle\check{\lambda}-w_\gamma(\check{\lambda})-\check{\gamma},\rho_{P(X)}\rangle}(\delta_{w_\gamma(\check{\lambda})}-\delta_{w_\gamma(\check{\lambda}-\check{\gamma})})\right).
    \end{equation*}
    Note that we can write the above exponent as
    \begin{equation*}
        \langle\check{\lambda}-w_\gamma(\check{\lambda})-\check{\gamma},\rho_{P(X)}\rangle=\langle\check{\gamma},\rho_{P(X)}\rangle\cdot\left(\langle\check{\lambda},\gamma\rangle-1\right).
    \end{equation*}
\end{remark}
\begin{example}
    In the case $X=\mathrm{Sp}_{2r}\backslash\mathrm{GL}_{2r},$ we have normalized spherical roots of type $G$
    \begin{equation*}
        \gamma_i=\alpha_{2i-1}+2\alpha_{2i}+\alpha_{2i+1},\quad 1\le i<r
    \end{equation*}
    where $\alpha_i=e_i-e_{i+1},$ and
    \begin{equation*}
        \rho_{P(X)}=(r-1)(e_1+e_2)+(r-3)(e_3+e_4)+\cdots+(1-r)(e_{2r-1}+e_{2r}).
    \end{equation*}
    Here $\Lambda\twoheadrightarrow\Lambda_X$ is the quotient $\bigoplus_{i=1}^{2r}\Z \check{e}_i\to\bigoplus_{i=1}^r\Z \check{e}_i'$ where $\check{e}_{2i-1},\check{e}_{2i}\mapsto \check{e}_i'.$ We have $\check{\gamma}_i=\check{e}_i'-\check{e}_{i+1}',$ and thus $\langle \check{\gamma}_i,\rho_{P(X)}\rangle=2.$ Thus if $\check{\lambda}=\sum_{i=1}^ra_i\check{e}_i',$ we have 
    \begin{equation*}
        q_0^{\langle\check{\lambda},\rho_{P(X)}\rangle}\mathrm{Rel}_\delta(\check{\lambda},\gamma_i)=\delta_{\check{\lambda}}-\delta_{\check{\lambda}-\check{\gamma}_i}-q_0^{2(a_{i+1}-a_i-1)}(\delta_{w_\gamma(\check{\lambda})}-\delta_{w_\gamma(\check{\lambda}-\check{\gamma}_i)})
    \end{equation*}
    which recovers \Cref{Reldef} in case the \SPgen.
\end{example}

\begin{example}\label{InvSat}
    Consider the case $X=H\backslash G$ where $(G,H)=(\mathrm{Res}_{F/F_0}\mathrm{GL}_r,U_r)$ with respect to $F/F_0$ an unramified quadratic extension. We denote $q_0$ the residue cardinality of $F_0.$ This case is not covered by \cite{Sakellaridis}, since $G$ is not split. Note that $X_F\defeq X\times_{F_0}F=\mathrm{GL}_r\backslash(\mathrm{GL}_r\times\mathrm{GL}_r)$ is the group case, and thus $\Lambda_X=\Lambda_{\mathrm{GL}_r}.$ Running the above discussion backwards, we can deduce an inverse Satake transform from our straightening relations. We consider the basis $e^{\check{\gamma}}$ of $\C[\Lambda_{\mathrm{GL}_r}]$ given by
    \begin{equation*}
        e^{\check{\gamma}}\defeq (-q_0)^{\langle\check{\lambda},\rho_{\mathrm{GL}_r}\rangle}\cdot\mathrm{char}(x_{\check{\lambda}}K)
    \end{equation*}
    under the relative Cartan decomposition $\mathcal{S}(X/K,\C)=\C[\Lambda_{\mathrm{GL}_r}^+].$ Then we have a relative Satake isomorphism
    \begin{equation*}
        \mathrm{Sat}\colon\mathcal{S}(X/K,\C)\simeq \C[\Lambda_{\mathrm{GL}_r}]^{W_{\mathrm{GL}_r}}
    \end{equation*}
    such that for $h\in\C[\Lambda_{\mathrm{GL}_r}]^{W_{\mathrm{GL}_r}}$ we have
    \begin{equation*}
        \mathrm{Sat}^{-1}(h)=\left(h\cdot L_X^{\frac{1}{2}}\right)\rvert_{\Lambda_{\mathrm{GL}_r}^+}
    \end{equation*}
    where
    \begin{equation*}
        L_X^{\frac{1}{2}}=\prod_{\check{\gamma}\in\check{\Phi}_{\mathrm{GL}_r}^+}\frac{1-e^{\check{\gamma}}}{1+q_0^{-1}e^{\check{\gamma}}}.
    \end{equation*}
    Moreover, by the discussion of \Cref{RSat-uH}, the map $\mathrm{Sat}$ is equivariant for the action of $\mathcal{H}(G,K,\C),$ where the action on the right hand side is given by the base change map
    \begin{equation*}
        \mathcal{H}(G,K,\C)\hookrightarrow\mathcal{H}(\mathrm{GL}_r(F_0),\mathrm{GL}_r(\O_{F_0}),\C)\iso\C[\Lambda_{\mathrm{GL}_r}]^{W_{\mathrm{GL}_r}}.
    \end{equation*}

\end{example}

%% file: structural/Bibliography.tex
\begingroup
\bibliographystyle{alpha}
\bibliography{structural/references}
\endgroup